\documentclass[12pt]{amsart}

\usepackage{amssymb,amsmath,amscd,amsfonts,amsthm,dsfont,color,graphicx,ulem}
\usepackage[latin1]{inputenc}

\theoremstyle{plain}
\newtheorem{lemma}{Lemma}

\newtheorem{theorem}[lemma]{Theorem}
\newtheorem{corollary}[lemma]{Corollary}

\theoremstyle{definition}

\newtheorem{remark}{Remark}

\newcommand{\N}{\mathbb{N}}

\newcommand{\R}{\mathbb{R}}
\newcommand{\C}{\mathbb{C}}
\newcommand{\CL}{\mathcal{L}}
\newcommand{\CM}{\mathcal{M}}
\newcommand{\dom}{\operatorname{Dom}}
\newcommand{\Span}{\operatorname{Span}}
\newcommand{\Spec}{\operatorname{Spec}}
\newcommand{\dist}{\operatorname{dist}}
\newcommand{\ud}{\,\mathrm{d}}
\renewcommand{\Re}{\operatorname{Re}}
\renewcommand{\Im}{\operatorname{Im}}

\numberwithin{equation}{section}

\def \p{\partial}

\def \e{\varepsilon}

\def \rmi{{\rm i}}

\title[Non-variational computation of eigenstates]{Non-variational computation \\ of the eigenstates 
of {D}irac operators \\ with radially symmetric potentials}

\begin{document}

\author{Lyonell Boulton}
\address{Department of Mathematics and Maxwell Institute for Mathematical 
Sciences, Heriot-Watt University, Edinburgh
EH14 4AS, United Kingdom}
\email{L.Boulton@hw.ac.uk}

\author{Nabile Boussaid}
\address{D\'epartement de Math\'ematiques, UFR Sciences et techniques - 
16, route de Gray - 25 030, Besan\c{c}on 
cedex, France}
\email{nabile.boussaid@univ-fcomte.fr}
\subjclass[2000]{}
\date{July 2008}

\begin{abstract}
We discuss a novel strategy for computing the eigenvalues and eigenfunctions of the relativistic 
Dirac operator with a radially symmetric potential. 
The virtues of this strategy lie on the fact that it avoids
completely the phenomenon of spectral pollution and it 
always provides two-side estimates for the eigenvalues 
with explicit error bounds on both eigenvalues and eigenfunctions.  
We also discuss convergence rates of the method as well as
illustrate our results with various numerical experiments. 
\end{abstract}

\maketitle


\section{Introduction}
\label{section:intro}

The free Dirac operator acting on 4-spinors of
$L^2(\R^3)^4$ is determined by the first order differential expression
\begin{equation*}
D:=\alpha\cdot P+\beta=
-\rmi\sum_{k=1}^3\alpha_k\partial_k + \beta,
\end{equation*}
where $\alpha=\left(\alpha_1,\alpha_2,\alpha_3\right)$, 
and the Pauli-Dirac matrices are:
\begin{gather*}
\alpha_i=\left(
\begin{array}{cc}
0&\sigma_i\\
\sigma_i&0
\end{array} \right)\quad\mbox{ and }\quad
\beta=\left(\begin{array}{cc}
I_{\C^2}&0\\
0&-I_{\C^2}
\end{array} \right),\\
\mbox{ for  } \sigma_1=\left(
\begin{array}{cc}
0&\; 1\\
1&\;0
\end{array} \right),\quad \quad\sigma_2=\left(
\begin{array}{cc}
0&-\rmi\\
\rmi&0
\end{array} \right)\quad\mbox{ and } \quad\sigma_3=\left(
\begin{array}{cc}
1&0\\
0&-1
\end{array} \right).
\end{gather*}
We assume that the units are fixed so that $m=c=\hbar=1$. 
Standard arguments involving the Fourier transform show that $D$ 
defines a self-adjoint operator with domain $H^1(\R^3)^4$ and that
the spectrum of $D$ is
\[
        \Spec(D)=(-\infty,-1]\cup[1,\infty).
\] 

Spherically symmetric potentials are Hermitean $4\times 4$ matrix 
multiplication operators, $V$, acting on $L^2(\R^2)^4$, such that ${\mathcal C}^\infty_0
(\R^3\setminus\{0\})\subset \dom(V)$ and 
\begin{equation*}
e^{\rmi \varphi n\cdot S}V(R^{-1}x)e^{-\rmi \varphi n\cdot S}=
V(x),\;\forall x\in \R^3,\; \forall \varphi\in [0,4\pi),
\end{equation*}  
where 
\begin{equation*}
S=\frac{1}{2}\begin{pmatrix}
0&\sigma\\
\sigma&0
\end{pmatrix}
\end{equation*}  
is the spin operator, and  $R$ is the matrix of the rotation of angle 
$\varphi$ and axis $n$. Here $\dom(V)$ denotes the maximal domain 
of $V$.

Spherically symmetric potentials may be constructed 
from  maps $\phi_{\mathrm{sc},\mathrm{el},\mathrm{am}}:\R \longrightarrow \R$ via
\begin{equation}
V(x)=\phi_{\mathrm{sc}}(|x|)\beta+\phi_{\mathrm{el}}(|x|)I_{\C^4}+i \phi_{\mathrm{am}}(|x|)
\beta \alpha\cdot \frac{x}{|x|}. \label{Eq:Example} 
\end{equation}  
The subscripts ``$\mathrm{sc}$'', ``$\mathrm{el}$'' and ``$\mathrm{am}$'', stand for ``scalar'', 
``electric'' and ``magnetic'' potential, respectively. Radial symmetry 
on the electric potential, for instance, is a consequence of the 
assumption that the atomic nucleus is pointwise and the electric 
forces are isotropic in an isotropic medium like the vacuum.
In the particular case $\phi_{\mathrm{sc}}=\phi_{\mathrm{am}}=0$ and $\phi_{\mathrm{el}}(r)=\gamma/r$, $|\gamma|<\sqrt{3}/2$,
$H$ describes the motion of a relativistic electron
in the field created by an atomic nucleus.

If $V$ is subject to suitable smallness and regularity conditions (such as  those ensuring that $V$ is relatively compact with respect to $D$), then $H:=D+V$ defines an essentially self-adjoint operator in $C^\infty_0(\R^3\setminus\{0\})^4$ with self-adjoint extension domain $H^1(\R^3)^4$ and 
\begin{equation} \label{eq:spess}
        \Spec_{\text{ess}}(H)=\Spec(D)=(-\infty,-1]\cup[1,\infty),
\end{equation}
see \cite[Theorems~4.2 and 4.16]{Thaller}. In fact there are known conditions, satisfied by potentials of practical interest (see e.g. \cite[Theorems~6 and A]{BerthierGeorgescu}), preventing the existence of embedded eigenvalues.

The addition of a non-zero potential might give rise to a non-empty discrete spectrum in the gap $(-1,1)$. These eigenvalues can only be found explicitly for a few simple systems which do not go much beyond the case of a single electron embedded in a field created by an atomic nucleus,
see \eqref{eq:eva}. For more complicated potentials, one has to rely either on asymptotic techniques (cf. \cite{Lap_Bir}, \cite{Gries}, \cite{Schmidt} and references therein) or on numerical estimations. 

Few robust computational procedures are currently available to estimate numerically the eigenvalues of $H$, see \cite{Dyall}, \cite{DES00}, \cite{DESV}, \linebreak \cite{DES03}, \cite{DEL}, \cite{LLT} and references therein. Since $H$ is strongly indefinite, a direct implementation of the projection method is not possible due to variational collapse. Under no further restriction on the potential or the reduction basis, accumulation points of eigenvalues of the finite dimensional approximate operator which do not belong to $\Spec(H)$ might appear, see \cite{Stanton} and \cite{Dyall}. This phenomenon is known as spectral pollution.

The aim of the present paper is to investigate the applicability of the so called quadratic projection method for finding eigenvalues and eigenfunctions of $H$. This method has been recently studied in an abstract setting (see \cite{Shargo}, \cite{Lev_Shar}, \cite{Bou1} and \cite{Bou2}) and it has already been applied with some success to crystalline Schr\"odinger operators, \cite{Bou_Lev}, and magnetohydrodynamics, \cite{Strauss}. In this approach, explained at length in Section~\ref{section:qpm}, the underlying discretised eigenvalue problem is quadratic in the spectral parameter (rather than linear), and has non-real eigenvalues.
Its main advantage over a standard projection method is its robustness:  it \textit{never} pollutes and it always provides 
\textit{a posteriori two-sided estimates} of the error of computed eigenvalues.

Section~\ref{subsection:sosp} is devoted to a self-contained description of the quadratic projection method. In Theorem~\ref{non_pollution} we present an alternative proof of Shargorodsky's non-pollution Theorem 
\cite[Theorem~2.6]{Lev_Shar}. Additionally we show that information about eigenfunctions can be recovered from the quadratic projection method, see \eqref{eq_approx_efu}. 

In Section~\ref{section:numerical} we test the practical applicability of the numerical scheme proposed in Section~\ref{section:qpm} by reporting on various numerical experiments.
As benchmark potentials we have chosen the purely coulombic, sub-coulombic and inverse harmonic electric potentials. 
In order to perform these numerical experiments, we split the space into upper and lower spinor component, after writing the problem in radial form. Spectral pollution in the standard projection method using this decomposition and the consequences of unbalancing the number of upper/lower components is discussed in Section~\ref{section:kine_bal}. 

We have chosen a basis of Hermite functions to reduce the continuous problem into finite-dimensional from. In Section~\ref{subsection:conv}
we perform a rigourous convergence analysis of the quadratic method 
using this basis. This analysis relies upon the general result \cite[Theorem~2.1]{Bou2}. The surprising numerical evidence included in Section~\ref{subsection:kin_umbalance} strongly suggest that,
under appropriate circumstances, choosing a basis that heavily pollute the standard projection method, significantly improves convergence rates of the quadratic projection method.

We have posted a fully functional Matlab code for assembling the matrices  involved in the quadratic projection method for $V$ a coulombic potential in the Manchester NLEVP collection of non-linear eigenvalue problems \cite{bhms08}. See the permanent link \cite{nlevp}. In Appendix~\ref{appendix:matr_entries} we include details of the explicit calculation of all matrix coefficients.

\section{Spectral pollution and upper/lower spinor component balance}
\label{section:kine_bal}

Consider, for any $\Psi \in L^2(\R^3)^4$, the spherical coordinates
representation:
\begin{equation}\label{eq:radialform}
\psi(r,\theta,\phi)=r\Psi(r\sin(\theta)\sin(\phi),r\sin(\theta)
\cos(\phi),r\cos(\theta))
\end{equation}
where $(r,\theta,\phi)\in(0,\infty)\times[0,\pi)\times[-\pi,\pi)$. 
The map $\Psi\mapsto \psi$ is an isomorphism between 
$L^2(\R^3)^4$ and $L^2((0,\infty),dr)\otimes L^2(S^2)^4$. If we  
decompose $L^2(S^2)^4$ as the sum of the so called 
two-dimensional angular momentum subspaces ${\mathfrak K}_{m_j,\kappa_j}$, 
the partial wave subspaces are given by
\begin{equation*}
{\mathfrak H}_{m_j,\kappa_j}=L^2((0,\infty),dr)\otimes 
{\mathfrak K}_{m_j,\kappa_j},
\end{equation*} 
so that $L^2(\R^3)^4=\bigoplus {\mathfrak H}_{m_j,\kappa_j}$.
Without further mention, here and below we always assume that 
the indices $(m_j,\kappa_j)$ 
run over the set $m_j\in\{-j,\cdots,j\}$ and 
$\kappa_j\in\{\pm (j+\frac{1}{2})\}$, for $j\in\{\frac{2k+1}{2}:k\in \N\}$. 
 
The $r$ factor in 
\eqref{eq:radialform} renders a  Dirichlet boundary condition at $0$.
The dense subspaces $C^\infty_0(0,\infty)\otimes 
{\mathfrak K}_{m_j,\kappa_j}\subset {\mathfrak H}_{m_j,\kappa_j}$ 
are invariant under the action of $H$.
If $V$ is as in \eqref{Eq:Example}, then
$H\upharpoonright C^\infty_0(0,\infty)\otimes 
{\mathfrak K}_{m_j,\kappa_j}$ is unitary equivalent to
\begin{equation}
\label{Eq:RadialDiracMatrix}
H_{m_j,\kappa_j}:=\begin{pmatrix}
1+\phi_{\mathrm{sc}}(r)+\phi_{\mathrm{el}}(r)&-\frac{d}{dr}+\frac{\kappa_j}{r}+\phi_{\mathrm{am}}(r)\\
\frac{d}{dr}+\frac{\kappa_j}{r}+\phi_{\mathrm{am}}(r)&-1-\phi_{\mathrm{sc}}(r)+\phi_{\mathrm{el}}(r)
\end{pmatrix}.
\end{equation}
The operators 
$H_{m_j,\kappa_j}$ are essentially self-adjoint in  
$C^\infty_0(0,\infty)^2$ under suitable conditions on the potentials $\phi_{\mathrm{sc},\mathrm{el},\mathrm{am}}$. Then
\[
      \Spec(H)=\overline{\bigcup 
\Spec(H_{m_j,\kappa_j})}. 
\]
Below we often
suppress the sub-index $(m_j,\kappa_j)$ from operators and spaces, and only
write the index $\kappa\equiv \kappa_j$. Note that the eigenvalues of $H$ are degenerate and their multiplicity is at least $m_j$. By virtue of
\eqref{eq:spess},
\[
      \Spec_{\text{disc}}(H)=\bigcup \Spec_{\text{disc}}(H_{\kappa}).
\]
Since
$\Spec_{\text{ess}}(H_{\kappa})=(-\infty,-1]\cup[1,\infty)$,
$H_{\kappa}$ are strongly indefinite. 

Let us consider a heuristic approach to the problem of spectral pollution for $H_\kappa$ and the decomposition of $L^2(0,\infty)^2$ into upper and lower spinor components. For simplicity we assume that the potential is purely electric and attractive, $\phi_{\mathrm{sc}}(r)=\phi_{\mathrm{am}}(r)=0$ and $\phi_{\mathrm{el}}(r)<0$. 

The pair $(u,v)\in \dom(H_\kappa)$ is a wave function of $H_\kappa$ with associated eigenvalue $E\in (-1,1)$ if and only if 
\begin{equation}    \label{eq:exact}
   (\phi_{\mathrm{el}}+1-E)u +(-\partial_r+\frac{\kappa}{r})v =0 \qquad \text{and} \qquad (\partial_r+\frac{\kappa}{r})u +(\phi_{\mathrm{el}}-1-E)v =0.
\end{equation}
System \eqref{eq:exact} can be decoupled into
\begin{equation}  \label{eq:deco}
\qquad L_Eu=0 \qquad \text{and} \qquad
v=-(\phi_{\mathrm{el}}- 1-E)^{-1}\left(\partial_r+\frac{\kappa}{r}\right)u,
\end{equation}
where
\begin{align*}
    L_E&:=-\left(-\partial_r+\frac{\kappa}{r}\right)(\phi_{\mathrm{el}}- 1-E)^{-1}\left(\partial_r+\frac{\kappa}{r}\right)+(\phi_{\mathrm{el}}+1-E) \\
    &\geq 
    (1+E)^{-1}\left(\partial_r+\frac{\kappa}{r}\right)^\ast
    \left(\partial_r+\frac{\kappa}{r}\right) +\phi_{\mathrm{el}} +(1-E).
\end{align*}
If we assume $\phi_{\mathrm{el}}$ is relatively compact with respect to the non-negative
operator $(\partial_r+\frac{\kappa}{r})^\ast(\partial_r+\frac{\kappa}{r})$, 
\begin{equation} \label{eq:ellip}
    \min \Spec_{\text{ess}} L_E\geq (1-E)>0.
\end{equation}
Moreover, the expression for $v$ in \eqref{eq:deco} 
yields $u\not\equiv 0$ and 
$v\not\equiv 0$. Hence $E\in \Spec_{\text{disc}}H_\kappa$ if and only if 
$0$ is in the discrete spectrum of  $L_E$.

Let $\CM_N\subset L^2(0,\infty)$ be a nested family of finite-dimensional subspaces such that $\CM_N\subset \CM_{N+1}$,
\[
    \overline{\bigcup_{N\geq 1} \CM_N}=L^2(0,\infty)
\]
and $\CL_{NM}:=\CM_N\oplus \CM_M\subset \dom(H_\kappa)$. Let $P_N$ denote the orthogonal projection onto $\CM_N$, so that $P_N\to I$ in the strong sense. With this notation we wish to apply the projection method to
operator $H_\kappa$ with test spaces $\CL_{NM}$.

We first consider the case $N=M$. Let $(u_N,v_N)^t\in \CL_{NN}$ be a sequence normalised by $\|u_N\|^2+\|v_N\|^2=1$, such that
\begin{align}
   P_N[(\phi_{\mathrm{el}}+1-E_N)u_N+(-\partial_r+\frac{\kappa}{r})v_N] =0, \label{eq:pro1}\\
   P_N[(\partial_r+\frac{\kappa}{r})u_N+(\phi_{\mathrm{el}}-1-E_N)v_N] =0
   \label{eq:pro2}
\end{align}
for a suitable sequence $E_N\to \tilde{E} \in (-1,1)$. Let $\chi_N$ be such that
\[
     v_N=-(\phi_{\mathrm{el}}- 1-E_N)^{-1}\left(\partial_r+\frac{\kappa}{r}\right)u_N
     +\chi_N.
\]
By virtue of \eqref{eq:pro2}, 
\begin{equation}     \label{eq:have}
     P_N(\phi_{\mathrm{el}}-1-E_N)\chi_N=0.
\end{equation}   
If we where able to prove that
\begin{equation}     \label{eq:want}
      \|P_N (-\partial_r+\frac{\kappa}{r})\chi_N\| \to 0,
      \qquad N\to \infty,
\end{equation}
by substituting into \eqref{eq:pro1}, we would have $\|P_N L_{\tilde{E}}u_N \|\to 0$. Thus, by virtue of the min-max principle alongside with \eqref{eq:ellip}, we would have $0\in \Spec_{\text{disc}} L_{\tilde{E}}$ and so $\tilde{E}\in \Spec_{\text{disc}} H_\kappa$.

Unfortunately, \eqref{eq:want} can not be easily verified. Since $P_N\to I$ in the strong sense, $(\phi_{\mathrm{el}}-1-E_N)\chi_N\to 0$ in the weak sense. Since $\phi_{\mathrm{el}}<0$, $\chi_N\to 0$ also in the weak sense. Therefore we are certain that $(-\partial_r+\frac{\kappa}{r})\chi_N \to 0$ and hence
$P_N L_{\tilde{E}}u_N \to 0$ weakly. This does not imply in general  \eqref{eq:want}, it only gives indication that the latter is a sensible guess.  

The key idea behind the above heuristic argument motivates several pollution-free numerical methods for computing eigenvalues of $H_\kappa$, including the quadratic projection method discussed in the forthcoming section: as \eqref{eq:exact} might be vulnerable to spectral pollution in the gap $(-1,1)$ because \eqref{eq:want} is not necessarily guaranteed, we characterise the eigenvalues of this problem by testing whether $0$ is in the spectrum of an auxiliary operator. This auxiliary operator has essential spectrum in $\{\Re(z)>0\}$, so that a neighbourhood of $0$ is always protected against spectral pollution (see the discussion preceding \cite[Lemma~5]{Bou1}.) In the quadratic projection method, \eqref{eq:want} is substituted by \eqref{eq_approx_efu}.

\begin{remark} \label{desmethod}
A successful procedure for computing eigenvalues of the Dirac operator is the one developed by Dolbeault, Esteban and S\'er\'e, \cite{DES00} and \cite{DES03}. This procedure systematically implements the above idea. Multiplying by $u$ the left side equation of \eqref{eq:deco} and integrating in the space variable, gives $A(E)u=0$ for
\[
     A(\lambda)v:=\int_0^\infty \frac{|(r^\kappa v)'|^2}{r^{2\kappa}(1+\lambda-\phi_{\mathrm{el}})} +(\phi_{\mathrm{el}}+1-\lambda)|v|^2 \ud r.
\] 
Both terms inside the integral decrease in $\lambda$ so, if $v$ is regular enough, there is a unique $\lambda=\lambda(v)\in \R$ satisfying $A(\lambda)v=0$. Upper estimates for the eigenvalues of $H_\kappa$ are then found from those of the matrix corresponding to the $\lambda$-dependant form $A(\lambda)v$ restricted to $v\in \CM_n$. 
\end{remark}

Suppose now that $M:\N\longrightarrow \N$. Let 
$(u_N,v_N)^t\in \CL_{NM(N)}$ be a sequence normalised by $\|u_N\|^2+\|v_N\|^2=1$, such that \eqref{eq:pro1} holds true and $P_N$ is replaced in \eqref{eq:pro2} by $P_{M(N)}$. If $\lim_{N\to \infty}M(N)/N<1$, we are certainly less likely to obtain \eqref{eq:want} as $P_N$ is replaced in \eqref{eq:have} by $P_{M(N)}$. If, on the other hand, $\lim_{N\to \infty}M(N)/N>1$, then we would be more confident about obtaining \eqref{eq:want}
for symmetric reasons. In Section~\ref{subsection:kin_umbalance} we will present a series of numerical experiments supporting this heuristic argumentation. In particular, spectral pollution in the standard projection method appears to increase as $\lim_{N\to \infty}M(N)/N$ decreases (see figures~\ref{fig:evo_smooth_gamma_neg}) .

\pagebreak

\section{The quadratic projection method}   \label{section:qpm}           

\subsection{The second order relative spectrum} \label{subsection:sosp}
We now describe the basics of 
the quadratic projection method. In order to simplify the notation,
below and elsewhere $ G$ denotes a generic
self-adjoint operator with domain $\dom( G)$ acting on a Hilbert space.
One should think of $ G$ as being any of the $H_{\kappa}$ introduced in the
previous section. The inner product in this Hilbert space is 
denoted by $\langle \cdot,\cdot \rangle$ and the norm by $\|\cdot\|$.

Let $\mathcal{L}\subset \dom( G)$ be a subspace of finite dimension.
Assume that $\mathcal{L}=\Span\{b_1,\ldots,b_n\}$, where the vectors $b_j$ are 
linearly independent.  Let
\begin{equation} \label{eq_coeff_qpm}
\begin{gathered}
    K:=(\langle  G b_j, Gb_k\rangle)_{j,k=1}^n, \qquad
    L:=(\langle  G b_j,b_k\rangle)_{j,k=1}^n \\ \qquad \text{and} \qquad
    B:=(\langle b_j,b_k\rangle)_{j,k=1}^n.
\end{gathered}
\end{equation}
For $z\in\C$, let $Q(z):=Bz^2-2zL+K\in \C^{n\times n}$.
The aim of the quadratic projection method is to compute 
the so called  \textit{second order spectrum of $ G$ relative 
to $\mathcal{L}$}: 
\[
  \Spec_2( G,\mathcal{L}):=\Spec(Q)=
 \{\lambda\in \C\,:\, Q(\lambda)\underline{v}=0,\, 
\text{some}\ 0\not=\underline{v}\in \C^n\}.
\]
Since $B$ is a non-singular matrix and $\det Q(z)$ is a polynomial of degree
$2n$, $\Spec_2( G,\mathcal{L})$ consists of at most $2n$ points. These
points do not lie on the real line, 
except if $\mathcal{L}$ contains eigenvectors of $ G$. However, since $Q(z)^\ast=Q(\overline{z})$, 
\[
    \overline{\Spec_2( G,\mathcal{L})}=\Spec_2( G,\mathcal{L}).
\]
Approximation of the discrete spectrum of $ G$ using the second order spectrum
has been discussed in \cite{Lev_Shar}, \cite{Bou2}, \cite{Bou_Lev} and
the references therein.  

The following result establishes a crucial
connection between $\Spec( G)$ and $\Spec_2( G,\mathcal{L})$. Without 
further mention, we often identify the elements $v\in\mathcal{L}$ 
with corresponding 
$\underline{v}\in\C^n$ in the obvious manner:
\[
    v=\sum_{k=1}^n \langle v,b_j^\ast \rangle b_j\qquad \text{and}
    \qquad \underline{v}=(\langle v,b_j^\ast\rangle)_{j=1}^n,
\]
where $\{b_j^\ast\}$ is the basis conjugate to $\{b_j\}$. 
Note that if the $b_j$ are mutually orthogonal and $\|b_j\|=1$, 
then $b_j^\ast=b_j$. Below and elsewhere $\Pi_\mathcal{S}$ denotes
the orthogonal projection onto a subspace $\mathcal{S}\subset \dom( G)$.

\begin{theorem} \label{non_pollution}
Let $\mathcal{L}\subset\dom( G)$ be any finite-dimensional subspace.
If $\lambda\in\Spec_2( G,\mathcal{L})$, then
\begin{equation} \label{eq_non_poll}
   [\Re(\lambda)-|\Im(\lambda)|,\Re(\lambda)+|\Im(\lambda)|]
   \cap \Spec( G)\not=\varnothing.
\end{equation}
Moreover, suppose that $E$ is an isolated eigenvalue of $ G$ with associated eigenspace $\mathcal{E}\subset \dom( G)$. Let
\[
   d_E:=\dist(E,\Spec  G \setminus\{E\})=
    \min\{|E-x|\,:\, x\in\Spec  G,\, x\not=E \}.
\]
If
\begin{equation} \label{eq:unique}
 [\Re(\lambda)-|\Im(\lambda)|,\Re(\lambda)+|\Im(\lambda)|]
   \cap \Spec( G)=\{E\}
\end{equation}
and $Q(\lambda)\underline{v}=0$ for 
$0\not=\underline{v}\in \C^n$, then the corresponding $v\in \mathcal{L}$  satisfies
\begin{equation} \label{eq_approx_efu}
    \frac{\|v-\Pi_\mathcal{E} v\|}{\|v\|}\leq \frac{ |\Im \lambda|}{d_E}.
\end{equation}
\end{theorem}
\begin{proof} 
Let $\lambda\in\Spec_2( G,\mathcal{L})$ and assume that $\Im(\lambda)\not=0$. Since
\[
    Q(z)= (\langle (z- G)b_j,(\overline{z}- G)b_k \rangle)_{j,k=1}^n,
\]
$Q(\lambda)\underline{v}=0$ for non-trivial 
$\underline{v}\in \C^n$ if and only if
\begin{equation} \label{eq_weak_form}
    \langle (\lambda- G)v,(\overline{\lambda}- G)w \rangle=0 \qquad
    \forall w\in \mathcal{L}.
\end{equation}
For $u,w\in \mathcal{L}$ and $z=\mu+i\nu$ where $\mu,\nu\in\R$, we have
\[
    \langle (z- G)u,(\overline{z}- G)w\rangle=
   \langle (\mu- G)u,(\mu- G)w\rangle-\nu^2 \langle u,w \rangle
    -2i\nu \langle (\mu- G)u,w \rangle.
\]
In particular, if we take $w=v$ in $\eqref{eq_weak_form}$, we achieve
\[
    \|(\Re \lambda- G)v\|^2-|\Im \lambda|^2\|v\|^2-2i|\Im \lambda|
    \langle (\Re \lambda- G)v,v\rangle=0.
\]
Thus 
\begin{equation} \label{eq_minimiser}
      \frac{\|(\Re \lambda- G)v\|}{\|v\|}=|\Im \lambda |
\end{equation}
and
\begin{equation} \label{eq_minimiser2}
 |\Im \lambda|\langle (\Re \lambda- G)v,v\rangle=0.
\end{equation}
But recall that $ G= G^\ast$, so
\[
    \dist(x,\Spec  G)=\min_{u\in \dom( G)}\frac{\|(x- G)u\|}{\|u\|}
\]
for $x\in\R$. Therefore
\[
    \dist(\Re \lambda ,\Spec  G)\leq |\Im \lambda |,
\]
confirming \eqref{eq_non_poll}.

For the second part, assume that $\lambda$, $E$ and $\mathcal{E}$ 
are as in the hypothesis, and let $v$ satisfy \eqref{eq_weak_form} and hence 
\eqref{eq_minimiser}. Then
\begin{align*}
  \|(E- G)v\| &\leq |E-\Re \lambda |\|v\|+\|(\Re \lambda- G)v\| \\
    &\leq 2|\Im \lambda | \|v\|.
\end{align*}
Since
\[
    \dist (x,\Spec  G\setminus \{E\})=\min_{u\in \dom(H),\, u\perp \mathcal{E}}
   \frac{\|(x- G)u\|}{\|u\|},
\]
we have
\[
   \|(E- G)(I-\Pi_\mathcal{E})v\|\geq d_E\|(I-\Pi_\mathcal{E})v\|,
\]
so that
\[
   \|(I-\Pi_{\mathcal{E}})v\| \leq \frac{\|(E- G)v \|}{d_E}.
   \]

Now,
\begin{equation*}
  \|(E- G)v\|^2 =|E-\Re \lambda |^2\|v\|^2
+\|(\Re \lambda- G)v\|^2 
+ 2(E-\Re \lambda)\langle (\Re \lambda- G)v,v\rangle.
\end{equation*}
Thus, by \eqref{eq_minimiser} and \eqref{eq_minimiser2},
\begin{equation*}
  \|(E- G)v\|^2 =|E-\Re \lambda |^2\|v\|^2
+|\Im \lambda |^2\|v\|^2= |E-\lambda |^2\|v\|^2.
\end{equation*}
This gives 
\[
   \|(I-\Pi_{\mathcal{E}})v\| \leq 
   \frac{|\Im \lambda|}{d_E}\|v\|
\]
as needed.
\end{proof}

This theorem suggest a method for estimating $\Spec(G)$ from \linebreak $\Spec_2( G,\mathcal{L})$. We call this method the 
 \textit{quadratic projection method}. Choose a suitable 
$\mathcal{L}\subset \dom( G)$, find $Q(z)$ and compute 
$\Spec_2( G,\mathcal{L})$. Those $\lambda\in\Spec_2( G,\mathcal{L})$ which are
close to $\R$ will necessarily be close to $\Spec( G)$, with a two-sided
error given by $|\Im(\lambda)|$.  Moreover if $\lambda$ is close enough
to an isolated eigenvalue $E$ of $G$, then 
\begin{equation} \label{eq:iso}    
   |\Re \lambda - E|\leq|\Im \lambda|
\end{equation}
and a vector
$0\not=v\in\mathcal{L}$ such that
$Q(\lambda)\underline{v}=0$ approaches the eigenspace associated to
this eigenvalue with an error also determined by $|\Im(\lambda)|$. 
Note that there is no concern with the position of $E$ relative to 
the essential spectrum, or any semi-definitness condition imposed
on $ G$. The procedure is always free from spectral pollution.

\begin{remark} \label{remark:better}
A stronger statement implying the first part of Theorem~\ref{non_pollution} can be found in \cite[Lemma~5.2]{Lev_Shar}. 
In fact, for isolated points of the spectrum, the residual on the right of \eqref{eq:iso} can actually be improved to $\frac{2|\Im(\lambda)|^2}{d_E}$ 
for $|\Im(\lambda)|$ sufficiently small, see 
\cite[Corollary~2.6]{Bou_Lev}. However, note that this later estimate is less robust in the sense that $d_E$ is not known a priori.
\end{remark}

\subsection{The Hermite basis} \label{subsection:hermite_basis}
In the forthcoming sections we apply the quadratic 
projection method to $ G=H_{\kappa}$. To this end we construct
finite-dimensional subspaces $\mathcal{L}\subset \dom(H_\kappa)$ 
generated by Hermite functions. 

Let the odd-order Hermite functions be defined by
\[
  \Phi_k(r):= c_{2k+1}^{-1}h_{2k+1}(r)
e^{-\frac{r^2}{2}}, \qquad r\geq 0,
\]
where $h_{n}(r)$ are the Hermite polynomials and
$c_n=\sqrt{2^{n-1}n!\sqrt{\pi}}$ are normalisation constants.
Motivated by the results of \cite[Section~3.4]{Bou2}
on Schr\"odinger operators with band gap essential spectrum,
we choose
\begin{equation} \label{eq_testspaces}
   \mathcal{L}\equiv 
\mathcal{L}_{NM}:=\Span\left\{\begin{pmatrix} \Phi_1(r) \\ 0\end{pmatrix},
\ldots,\begin{pmatrix} \Phi_N(r)\\ 0\end{pmatrix},
\begin{pmatrix} 0\\ \Phi_1(r) \end{pmatrix}, \ldots
, \begin{pmatrix} 0\\ \Phi_M(r) \end{pmatrix}
 \right\}.
\end{equation}
Below we might consider an unbalance between the number of basis elements
in the first and second component, $N\not=M$. Without further mention,
we often write $\mathcal{L}_N\equiv \mathcal{L}_{NN}$, and 
$\mathcal{L}_n\equiv \mathcal{L}_{N(n),M(n)}$ when $M$ and $N$ depend upon $n$.

We now compute the matrix coefficients of $Q(z)$. For this
we recall some properties of $\Phi_k(r)$.
The Hermite polynomials are defined by the identity
\begin{equation*}
h_n(z) := (-1)^ne^{z^2}\frac{d^n}{dz^n}e^{-z^2}, \qquad n\in \N.
\end{equation*}
They satisfy the recursive formulae,
\begin{align}
		h'_n(z) &= 2nh_{n-1}(z), \label{eq:der}\\
		h_{n+1}(z) &= 2zh_n(z)-2nh_{n-1}(z) \label{eq:recu}
\end{align}
and form an orthogonal set on the interval
$(-\infty,\,\infty)$ with weight factor $e^{-z^2}$,
\[
   \int_{-\infty}^\infty h_m(z)h_n(z)e^{-z^2}\,dz = 2^n n!
   \sqrt{\pi}\delta_{nm}.
\]
The generating function of this family of polynomials is
$$\sum_{n=0}^\infty h_n(z)\frac{t^n}{n!}=e^{2zt-t^2}.$$
Thus
\begin{equation} \label{eq:her0}
\sum_{n=0}^\infty h_n(0)\frac{t^n}{n!}=\sum_{n=0}^\infty
\frac{(-1)^nt^{2n}}{ n!}, \qquad \text{so that} \qquad
h_{n}(0)=\left\{ \begin{array}{ll} 0 & n-\text{odd}, \\
\frac{(-1)^{n/2}n!}{ (n/2)!} & n-\text{even}. \end{array} \right.
\end{equation}

The odd-order Hermite functions  
are the normalised wave functions of a harmonic oscillator,
\begin{equation} \label{eq:eva_ho} 
   -\Phi_{k}''(r)+r^2\Phi_k(r)=(2k+1)\Phi_k(r),
\end{equation}
subject to Dirichlet boundary condition at the origin.
They form an orthonormal basis of $L^2(0,\infty)$
and so $B=I$ in \eqref{eq_coeff_qpm}. 

The entries of the 
matrices $K$ and $L$ in \eqref{eq_coeff_qpm}
can also be found explicitly
from known properties of the Hermite polynomials. The crucial terms  
for assembling these matrices are obtained in Appendix~\ref{appendix:matr_entries}.

\subsection{Convergence} \label{subsection:conv}
The procedure described above is useful,
provided we can find points of $\Spec_2(H_{\kappa},\mathcal{L})$ 
near the real axis. Here we formulate sufficient conditions on the sequence
of subspaces $\mathcal{L}_{n}$, in order to guarantee
the existence of a sequence $\lambda_n\in \Spec_2(H_{\kappa},\mathcal{L}_n)$ 
accumulating at points of the discrete spectrum of $H_{\kappa}$. We 
then show that the sequence of subspaces \eqref{eq_testspaces} 
satisfies these conditions.

We firstly recall the following result
\cite[Theorem~2.1]{Bou2}. 
Note that $\mathcal{L}_{NM}$ are subspaces of $\dom(H_{\kappa}^2)$ for all $\kappa$.

\begin{theorem} \label{approx}
Let $E$ be an isolated eigenvalue of finite multiplicity
of $ G$ with associated eigenspace denoted by $\mathcal{E}$.
Suppose that $\mathcal{L}_n\subset \dom ( G^2)$ is a sequence
of subspaces such that
\begin{equation} \label{eq_cond_approx}
   \| G^p(u-\Pi_n u)\|\leq \delta(n)\|u\| \hspace{1cm} 
\forall u\in\mathcal{E},\, p=0,1,2,
\end{equation}
where $\delta(n)\to 0$ as $n\to \infty$ is independent of $u$ and $p$.
Then there exists $b>0$ and $\lambda_n \in \Spec_2( G,\mathcal{L}_n)$, such that
\begin{equation} \label{convergence_rate}
      |\lambda_n- E|< b \delta(n)^{1/2}. 
\end{equation}
\end{theorem}

We now verify \eqref{eq_cond_approx} for $ G=H_{\kappa}$
and $\mathcal{L}_n$ as in \eqref{eq_testspaces}. To this end we consider 
an argument similar to the one discussed in \cite[Section~3.4]{Bou2} 
for the case of the cristaline non-relativistic Schr\"odinger operator.

Let 
\[
    A= \begin{pmatrix} \tilde{A} & 0 \\ 0 &   \tilde{A} 
   \end{pmatrix}
\]
where $\tilde{A}=-\partial_r^2+r^2$ acting on $L^2(0,\infty)$, 
subject to Dirichlet boundary conditions at the origin. We fix the domain of $A$ as
\[
  \dom(A)=\bigcup_{N\in\N} \mathcal{L}_N,
\]
so that $A=A^\ast$, $A$ has a compact resolvent  
and $\mathcal{L}_{N}\ominus \mathcal{L}_{N-1}$ are the eigenspaces
of $A$. 

For $f$ and $g$ regular enough, we have
\[
   H_{\kappa}\begin{pmatrix}f \\ g\end{pmatrix}=
  \begin{pmatrix}  p_1 f+ p_2 g \\  p_3 f+ p_4 g \end{pmatrix}
\qquad \text{and}\qquad
 H_{\kappa}^2\begin{pmatrix}f \\ g\end{pmatrix}=
  \begin{pmatrix}  q_1 f+ q_2 g \\  q_3 f+ q_4 g \end{pmatrix},
\]
where $p_j$ are linear polynomials and $q_j$ are quadratic polynomials
in the variables $\partial_r,\,\phi_{\mathrm{sc},\mathrm{el},\mathrm{am}}$ and $\kappa/r$.
Condition \eqref{eq_cond_approx} 
is achieved by showing that both $H_{\kappa}$ and 
$H_{\kappa}^2$ are relatively
bounded in the sense of operators with respect to $A$.

The following results can be easily extended to more general 
potentials. Here we only consider those of interest in our present discussion.

\begin{lemma} \label{rel_bound}
Suppose that   
\begin{equation} \label{eq:condpot}
       \phi_{\mathrm{sc},\mathrm{el},\mathrm{am}}=\chi_{\mathrm{sc},\mathrm{el},\mathrm{am}}+\psi_{\mathrm{sc},\mathrm{el},\mathrm{am}}
       \quad \text{where}\quad \left\{\begin{array}{l}|\chi_{\mathrm{sc},\mathrm{el},\mathrm{am}}(r)|
\leq c r^{-1},\, \forall r>0, \\ \psi_{\mathrm{sc},\mathrm{el},\mathrm{am}}\in L^\infty(0,\infty), \end{array}\right.
\end{equation}
for some constant $c>0$. Then 
$\dom(A)\subseteq \dom(H_{\kappa}^2)$ and there 
exist constants $a,b>0$ such that
\[
    \|H_{\kappa}^p v\|\leq a\|Av\|+b\|v\| \qquad \forall\,
    v\in \dom(A),\, p=0,1,2.
\]
\end{lemma}
\begin{proof}
It is enough to check that the $H_{\kappa}^p$ are
relatively bounded with respect to $ -\partial_r^2 I_{\C^2}$. 
This, on the other hand, is a straightforward consequence
of Hardy's inequality. 
\end{proof}

By combining this lemma with Theorem~\ref{approx} and
Theorem~\ref{non_pollution}, we achieve the following.

\begin{corollary} \label{corollary:conv}
Let $E$ be an isolated eigenvalue of $H_\kappa$ of finite 
multiplicity with associated eigenspace $\mathcal{E}$.
Let $\mathcal{L}_{NM}$ be defined by \eqref{eq_testspaces}. Suppose
that $\phi_{\mathrm{sc},\mathrm{el},\mathrm{am}}$ satisfy \eqref{eq:condpot} and assume 
additionally that  $\mathcal{E}\subset \dom(A^q)$ for some $q>1$. 
For $b>0$ large enough and independent of $N$ or $M$, we can always find 
a sequence $\lambda_{NM}\in \Spec_2(H_\kappa,\mathcal{L}_{NM})$, 
such that
\begin{equation} \label{eq:esti_conv}
\begin{gathered}
     |\lambda_{NM}-E|<b (N^{-\frac{q-1}{2}}+M^{-\frac{q-1}{2}}) \qquad \text{and} \\ 
      \|(I-\Pi_\mathcal{E})v_{NM}\|<b(N^{-\frac{q-1}{2}}+M^{-\frac{q-1}{2}}),
\end{gathered}
\end{equation}
where $v_{NM}\in\mathcal{L}_{MN}$ solves $Q(\lambda_{NM})
\underline{v_{NM}}=0$.
\end{corollary} 
\begin{proof}
Let $v=\begin{pmatrix} v_1 \\ v_2 \end{pmatrix}\in \mathcal{E}$ 
normalised by $\|v\|=1$. By virtue of Lemma~\ref{rel_bound} and Theorem~\ref{approx}, 
the desired conclusion follows if we are able to find $b>0$ such that
\begin{equation} \label{eq:app1}
   \sum_{k=n}^\infty |\langle \tilde{A}^r v_j,\Phi_k\rangle|^2<bn^{-2(q-r)}, 
\end{equation}
for $r=0,1$ and $j=1,2$. In order to show \eqref{eq:app1}, note that
the hypothesis $v\in \dom(A^q)$ ensures
\begin{align*}
   (2n)^{2(q-r)}\sum_{k=n}^\infty|\langle \tilde{A}^rv_j,\Phi_k\rangle |^2
   & \leq \sum_{k=n}^\infty (2k+1)^{2q}|\langle v_j,\Phi_k\rangle |^2
   \\ & =\sum_{k=n}^\infty |\langle \tilde{A}^qv_j,\Phi_k\rangle |^2
   \to 0
\end{align*}
as $n\to \infty$. \end{proof}

\begin{remark} \label{remark:disc}
Suppose that the number of degrees of freedom $M+N=2n$ is 
fixed. Modulo the constant $b$, the  bound on the right side of
\eqref{eq:esti_conv} is optimal at $(N,M)=(n,n)$. This suggests that
an optimal rate of approximation might be achieved by choosing an equal number of upper/lower spinor components in \eqref{eq_testspaces}. Contrary to this presumption, and   
depending on the potential $V$, the numerical evidence we present in 
sections~\ref{subsection:kin_umbalance} and \ref{subsection:conv_num},
show that the residual on 
the right side of \eqref{eq:iso} can in some cases decrease 
significantly (over $18\%$ in some cases) by suitably 
choosing $N\not= M$.
\end{remark}

We now explore precise conditions on the potential, in 
order to guarantee the hypothesis of Corollary~\ref{corollary:conv}. 

\begin{lemma}\label{lemma:regular}
Let  $\phi_{\mathrm{sc},\mathrm{el},\mathrm{am}}\in C^\infty(0,\infty)$ be such that $\phi_{\mathrm{sc},\mathrm{el},\mathrm{am}}(r)\to 0$
as $r\to \infty$. Assume additionally that $r\mapsto r^\alpha \phi_{\mathrm{sc},\mathrm{el},\mathrm{am}}(r)$ 
are locally bounded for some $\alpha\in(0,1)$. 
Let $H_\kappa u=Eu$.  For sufficiently small $a>0$,
\[
     \|e^{ar}u\|_{H^p(0,\infty)}<\infty \qquad \forall p\in\N.
\]
\end{lemma}
\begin{proof}
See \cite[Corollary 3.1]{BerthierGeorgescu}. Let 
\begin{equation*}
D_\kappa=\begin{pmatrix}
1&-\frac{d}{dr}+\frac{\kappa_j}{r}\\
\frac{d}{dr}+\frac{\kappa_j}{r}&-1
\end{pmatrix} \quad \text{and} \quad
V=\begin{pmatrix}
\phi_{\mathrm{sc}}+\phi_{\mathrm{el}}&\phi_{\mathrm{am}}\\
\phi_{\mathrm{am}}&-\phi_{\mathrm{sc}}+\phi_{\mathrm{el}}
\end{pmatrix}.
\end{equation*}
For any $0<\e<\min|E\pm 1|$, we can always separate $V=V_c+V_\e$ 
where:
\begin{itemize}
\item[{(a)}] $V_c$ is smooth, it has compact support and a singularity of order  $O(r^{-\alpha})$ at the origin,
\item[{(b)}]$V_\e$ is smooth with bounded derivatives and $\|V_\e\|_\infty<\e$.
\end{itemize}
Then
\begin{equation*}
 u=-(D_\kappa+V_\e-E)^{-1}V_cu.
\end{equation*}
Multiplying this identity by $D_\kappa^p$ and $e^{ar}$, yields 
\begin{equation} \label{eq:bst}
D_\kappa^p e^{ar}u=-D_\kappa^p(D_\kappa+V_\e+i a-E)^{-1}
e^{ar}V_cu.
\end{equation}

Let $\sigma\in\R$. On the one hand, $(D_\kappa+V_\e+i a-E)^{-1}$ is a bounded operator from $H^\sigma(0,\infty)$ to $H^{\sigma+1}(0,\infty)$. On the other hand, multiplication by
$V_c$ is a bounded operator from $H^\sigma(0,\infty)$ into $H^{\sigma-\alpha}(0,\infty)$.
Indeed, note that \cite[Theorem 2.1(i) for $p=q=2$]{Triebel} yield the latter result for $\sigma\in \N$ after commutation and iteration, then duality and interpolation ensure it for all $\sigma\in \R$. 
Therefore, as $(D_\kappa+V_\e+\rmi a-E)^{-1}V_c$ is bounded from $H^{\sigma}(0,\infty)$ to $H^{\sigma+(1-\alpha)}(0,\infty)$,
\eqref{eq:bst} and a standard bootstrap argument ensure the desired conclusion.
\end{proof}

We remark that, by using \cite[Lemma 5.1]{Hislop} (which can be  modified in order to get ride of the term $\frac{1}{2}$ in its assumption $(i)$), necessarily $a<\sqrt{(1-\e)^2-E^2}$ for the above lemma to hold true.


\begin{corollary} \label{corollary:approx_smooth}
Assume that $\phi_{\mathrm{sc},\mathrm{el},\mathrm{am}}$ are as in Lemma~\ref{lemma:regular}.
Let $E$ be an isolated eigenvalue of $H_\kappa$ of finite multiplicity with associated eigenspace $\mathcal{E}$.
Let $\mathcal{L}_{NM}$ be defined by \eqref{eq_testspaces}. For $b>0$ large enough and independent of $N$ or $M$, we can always find 
a sequence $\lambda_{NM}\in \Spec_2(H_\kappa,\mathcal{L}_{NM})$, 
such that \eqref{eq:esti_conv} holds true if $q<5/4$.
\end{corollary}
\begin{proof}
Let $v\in\mathcal{E}$ be as in the proof of 
Corollary~\ref{corollary:conv}. We show that
\begin{equation} \label{eq:app2}
   \sum _{k=0}^\infty (2k+1)^{2q}|\langle 
   v_j,\Phi_k\rangle|^2<\infty,
\end{equation}
for $j=1,2$. Let $F=(\partial_r-r)$. Integration by parts and
\eqref{eq:der}, yield 
\begin{equation*} \label{eq:app3}
\begin{aligned}
   \langle & v_j,\Phi_k\rangle  = \frac{1}{c_{2k+1}}
\int_0^\infty v_j(r)h_{2k+1}(r)e^{-r^2}{2} \ud r \\
& = \frac{1}{2(2k+2)c_{2k+1}}
\int_0^\infty v_j(r)h_{2k+2}'(r)e^{-\frac{r^2}{2}} \ud r  \\
&= \frac{-1}{2(2k+2)c_{2k+1}}
\int_0^\infty Fv_j(r)h_{2k+2}(r)e^{-\frac{r^2}{2}} \ud r  \\
&= \frac{1}{2^2(2k+2)(2k+3)c_{2k+1}}
\int_0^\infty Fv_j(r)h_{2k+3}'(r)e^{-\frac{r^2}{2}} \ud r  \\
&= \frac{1}{2^2(2k+2)(2k+3)c_{2k+1}}
\int_0^\infty F^2v_j(r)h_{2k+3}(r)e^{-\frac{r^2}{2}} \ud r  \\
&= \frac{-1}{2^3(2k+2)(2k+3)(2k+4)c_{2k+1}}\left(
F^2v_j(0)h_{2k+4}(0)+
\int_0^\infty F^3v_j(r)h_{2k+4}(r)e^{-\frac{r^2}{2}} \ud r\right)\\
&=\frac{-F^2v_j(0)h_{2k+4}(0)}{2^3(2k+2)(2k+3)(2k+4)c_{2k+1}}
+\frac{\int_0^\infty F^4v_j(r)h_{2k+5}(r)e^{-\frac{r^2}{2}} \ud r}{2^4(2k+2)\ldots(2k+5)c_{2k+1}}=a_1+a_2.
\end{aligned}
\end{equation*}
Identity \eqref{eq:her0} alongside with the fact that $|F^2v_j(0)|<\infty$ (see Lemma~\ref{lemma:regular}) and the Stirling formula, ensures that $a_1\sim n^{-7/4}$ as $n\to \infty$.  On the other hand, 
Lemma~\ref{lemma:regular} ensures that $F^4v_j\in L^2(0,\infty)$. Thus,
since
\[
     \frac{c_{2k+5}}{2^4(2k+2)\ldots(2k+5)c_{2k+1}}\sim
     n^{-2},
\]
$a_2=O(n^{-2})$ as $n\to\infty$. This guarantees \eqref{eq:app2}
for $q<5/4$.
\end{proof}

According to this corollary, for smooth potentials, the order of approximation of the quadratic projection method
to any eigenvalue $E$ of $H_\kappa$ should be at least a power
$1/8$ of the dimension of $\mathcal{L}_{MN}$. The numerical experiment performed in Section~\ref{subsection:conv} show that this bound
improves substantially for particular potentials. See Figure~\ref{fig:rate_unbalance} (right) and Table~\ref{table:expo1} (right).

\begin{remark} \label{remark:speed}
The arguments involved in the proof of this theorem show that, the behaviour of the wave functions at the singularity and its regularity are the main ingredients responsible for controlling the speed of approximation when using a Hermite basis \eqref{eq_testspaces}. 
\end{remark}

\section{Some numerical experiments}
\label{section:numerical}

We now report on various numerical experiments performed for very simple radially symmetric potentials. It is not our intention to show
accurate computations, but rather to illustrate how the method discussed
in Section~\ref{section:qpm} 
can be implemented in order to rigourously enclose eigenvalues and
compute eigenfunctions of the Dirac operator.

\subsection{Ground state of the purely coulombic potential}
\label{subsection:coul}
We begin by considering the analytically solvable case of $V$ being
a radially symmetric purely coulombic potential: $\phi_{\mathrm{sc}}=\phi_{\mathrm{am}}=0$, 
$\phi_{\mathrm{el}}(r)=\gamma/r$. Here $-\sqrt{3}/2<\gamma<0$.
The eigenvalues of $H_{\kappa}$ are given explicitly
by
\begin{equation} \label{eq:eva}
   E_j=\left( 1+\frac{\gamma^2}{(j+\sqrt{\kappa^2-\gamma^2})^2}  \right)^{-1/2}.
\end{equation}
Note that $E_j\to 1$ as $j\to\infty$ for all values of $\kappa$.
The ground state of the full coulombic Dirac operator $H$ is achieved
when $\kappa=-1$ and $j=0$.

In Figure~\ref{fig:1} we superimpose the computation of 
$\Spec_2(H_{-1},\mathcal{L}_n)$ for two values of $n$,
in a narrow box near the interval $[-3,3]$. Here $\mathcal{L}_n$ is given 
by \eqref{eq_testspaces}. For the set of parameters considered 
($\kappa=-1$ and $\gamma=-1/2$), \eqref{eq:eva} yields
$E_0\approx 0.866025$, $E_1\approx 0.965925$, $E_2\approx 0.9851210$,
$E_3\approx 0.99174012$ and $E_4\approx 0.9947623$. 
A two sided approximation of $E_0$ is achieved
from the point $\lambda\in \Spec_2(H_{-1},\mathcal{L}_{1000})$ at
$\lambda\approx 0.8661+0.0236i$. According to \eqref{eq_non_poll}, 
there should be an eigenvalue of $H_{-1}$ in the interval 
$[0.8661-0.0236,0.8661+0.0236]$. This eigenvalue happens to be $E_0$.
For $E_1$, $E_2$ and the pair ($E_3$, $E_4$), we can also derive similar
conclusions. Note that 
$\Spec_{\text{ess}}(H_{-1})$ is also revealed by 
points of $\Spec_2(H_{-1},\mathcal{L}_n)$ seemingly accumulating 
at $(-\infty,-1]\cup[1,\infty)$.

In Figure~\ref{fig:2} we show approximation of the corresponding ground
wave function associated with $E_0$. We have also depicted the
analytical eigenfunction:
\begin{equation} \label{eq:ge}
            u_0(r)=\nu_0\begin{pmatrix} \gamma \\ (1-\gamma^2)^{1/2}-1 
\end{pmatrix}
   r^{\sqrt{1-\gamma^2}}e^{-(\gamma E_0/\sqrt{1-\gamma^2})r},
\end{equation}
where $\nu_0$ is chosen so that $\|u_0\|=1$.
From this picture it is clear that, at least qualitatively, $u_0(r)$ seems to be captured quite well even for small values of $n$. 

We show a quantitative analysis of the calculation of $u_0$ in 
Table~\ref{table:1}. In the middle column we compute the residual 
on the left of \eqref{eq_approx_efu} and on the last column we 
compute the right hand side of  \eqref{eq_approx_efu}. It is quite 
remarkable that the actual residuals are over $74\%$ smaller than the 
error predicted by Theorem~\ref{non_pollution}.

\subsection{$\beta$-Dependence of the sub-coulombic potential}
\label{subsection:sub_coul}
We now investigate the case of the potential being
radially symmetric and sub-coulombic: $\phi_{\mathrm{sc}}=\phi_{\mathrm{am}}=0$, 
$\phi_{\mathrm{el}}(r)=\gamma/r^{\beta}$ for $\beta\in(0,1)$. 
Here $-1<\gamma<0$. The purpose of this experiment is to 
show how Theorem~\ref{non_pollution}
provides a priori information about $\Spec_{\text{disc}}(H_\kappa)$ even for small values of $n$. Note that \eqref{eq:spess} is guaranteed from \cite[Theorems~4.7 and 4.17]{Thaller}. Furthermore $H$ has infinitely many eigenvalues according to \cite[Theorem~4.23]{Thaller}.

In Figure~\ref{fig:3}, we show computation of the ground state of $H_{-1}$,
for $\beta=0.1:0.1:1$ and $\gamma=-1/2$. As $\beta\to 1^-$, 
$E_0\to 0.8661$, the ground eigenvalue of the coulombic Dirac operator.
As $\beta\to 0^+$ the eigenvalue remains above $1/2$. Note that the
family of operators $H_{\kappa}$ is not analytic at $\beta=0$ for this potential.
For $\beta=0$ the spectrum becomes
\[
   \Spec(H_\kappa)=(-\infty,-3/2]\cup[1/2,\infty).
\]
The vertical bars show $|\Im(\lambda)|$, the maximum error in the computation of $E_0\approx \Re(\lambda)$ given by
Theorem~\ref{non_pollution}. For this example we have chosen $n=15$. 
Table~\ref{table:2} contains the data depicted in Figure~\ref{fig:2}.
Observe that the error increases as $\beta\to 0^+$ and $\beta \to 1^-$.
This seems to be a consequence of the fact that $E_0$ becomes closer
to other spectral points at both limits, so $d_{E_0}\to 0$.

\subsection{The inverse harmonic electric potential} 
\label{subsection:smooth_pot}
In this set of experiments we consider another canonical 
example in the theory of Dirac operators: $\phi_{\mathrm{sc}}=\phi_{\mathrm{am}}=0$, 
$\phi_{\mathrm{el}}(r)=\gamma/(1+r^2)$ for $\gamma<0$. The discrete spectrum of $H$ is known to be finite for $-1/8<\gamma<0$ and infinite for $\gamma<-1/8$, \cite{Klaus}. As the parameter $\gamma$ decreases, we expect that eigenvalues will appear at the threshold 1, move through the gap, and leave it at -1. This dynamics is shown in Figure~\ref{fig:evo_smooth_gamma_neg}, for the ground eigenvalue of $H_{-1}$.
It is a long standing question whether the eigenvalues become resonances when they re-enter the spectrum.

In Figure~\ref{fig:eigenfu_smooth} we depict the
first three eigenfunctions of $H_{-1}$ for
$\gamma=-4$. They correspond to the
eigenvalues $E_0\approx -0.3955$, $E_1\approx0.6049$ and $E_2\approx 0.9328$. See also Figure~\ref{fig:evo_unbalance}.
Note that for $\gamma=-4$, both components of
the eigenfunctions appear to obey a Sturm-Liouville type
oscillation hierarchy.

\subsection{Upper/lower spinor component balance and approximation of eigenvalues}
\label{subsection:kin_umbalance} 

We now investigate the effects of ``unbalancing'' the basis by choosing $N\not=M$.

In Figure~\ref{fig:unbalance_coulomb}, we have performed the following experiment. Fix the number of degrees of freedom, $\dim(\mathcal{L}_{MN})=200$. Then for $N=10:5:190$ and $M=200-N$, use the quadratic method as well as the Galerkin method to approximate eigenvalues of $H_\kappa$ in the spectral gap $(-1,1)$. We firstly consider $\phi_{\mathrm{sc}}=\phi_{\mathrm{am}}=0$ and $\phi_{\mathrm{el}}(r)=-1/(2r)$. 

The Galerkin method might or might not produce spurious eigenvalues. 
The quadratic method will always provide two-sided non-polluted bounds for the true eigenvalues with a residual, obtained from \eqref{eq_non_poll}, which might change with $N$. See also figures~
\ref{fig:evo_smooth_gamma_neg} and \ref{fig:unbalance_coulomb}. 
The Galerkin method appears to pollute heavily near
the upper end of the gap for $N>M$, as predicted by the considerations of Section~\ref{section:kine_bal}. Moreover,  for the ground state, the minimal $|\Im(\lambda)|$ is not achieved at $N=100$ which corresponds to $N=M$, but rather at some $N>100$. It is remarkable that the residual are reduced significantly (up to $66\%$ for the true residual) when $M(N)/N\approx 1/5$.

If we performed the analogous experiment for the inverse harmonic potential, the conclusion are also rather surprising. See Figure~\ref{fig:evo_unbalance}. The Galerkin method appears to pollute heavily near the upper end of the gap for $N>M$ as predicted in Section~\ref{section:kine_bal}. However, now the approximation is improved by over $16\%$ for $E_0$ and over $18\%$ for $E_1$, if $M(N)/N\approx 3$.

We can explain these phenomena by considering the relation between the components of the exact eigenvectors. 

In the case of a purely coulombic potential, the ground state is given by \eqref{eq:ge} where $\nu_0$ is a real constant. The lower spinor component just differs from the upper one by a scalar factor. When $\gamma\in(0,1)$, the lower component is smaller in modulus than the upper one. Choosing $N>M$, can reduce an upper bound of the residual associated to the first component, while the residual associated to the second component remains small due to the smallness of the lower component. 

In the case of an inverse harmonic purely electric potential, this argument fails, as the two spinor components of the eigenfunction are not a scalar factor of each other, see Figure~\ref{fig:eigenfu_smooth}. If we denote an eigenfunction by $u=(u^{\mathrm{upp}},u^{\mathrm{low}})$,
the figure suggests that $|\partial_r^2 u^{\mathrm{upp}}(0)|<<|\partial_r^2 u^{\mathrm{low}}(0)|$. As $|\partial_r^2\Phi_k(0)|=0$, it is natural to expect that a decrease in the residual is only achieved by choosing a suitable $M>N$.

\begin{remark} \label{remark:nonpo}
Although we can not prove it rigourously, strong evidence suggests that (for any of the potentials considered above) no spurious eigenvalue is produced by the Galerkin method when $N=M$. Why bothering then with more complicated procedures, such as the quadratic projection method, 
to avoid inexistent spectral pollution. A partial answer is, on the one hand, robustness: we do not know a priori whether the Galerkin method pollutes for a given basis. On the other hand, as the experiments of this section suggest, some times forcing a kinetic unbalance into a model might improve convergence properties. 
\end{remark}

\subsection{Convergence properties of the odd Hermite basis}
\label{subsection:conv_num} 
A convergence analysis, as the number of degrees of freedom increases, can be found in Figure~\ref{fig:rate_unbalance} and Table~\ref{table:expo1}.
Due to the discussion of Section~\ref{subsection:kin_umbalance}, we consider this for different ratios $N/M$. 

The right graph shows that the conclusion of Corollary~\ref{corollary:approx_smooth} is far from optimal
for the inverse harmonic potential of 
Section~\ref{subsection:smooth_pot}. As expected from 
Section~\ref{subsection:kin_umbalance}, a faster convergence rate as well as smaller residuals are found if we suitably choose $N<M$.

The left graph corresponds to the coulombic potential discussed in
Section~\ref{subsection:coul}. It clearly shows that the 
order of convergence of $\lambda_n\to E$ does not obey the estimate $|\lambda_n-E|\leq O(n^{-a})$ (for some $a>0$) of Corollary~\ref{corollary:conv}. In fact the convergence rate seems to decrease as we increase the number of degrees of freedom. 
This reduction in the speed of convergence can be prevented by putting $M=f(N)$ for a suitable non-linear increasing function $0<f(x)<x$. The optimal $f(x)$, however, might depend on the eigenvalue to be approximated.

\begin{remark} 
According to Remark~\ref{remark:better}, the actual approximate eigenvalue $\Re(\lambda)$ is correct up to $O(n^{-2a})$, where $a$ is the second  column of Table~\ref{table:expo1}. Furthermore, note that, in the case of the coulombic potential we can compute directly the true residual $|\Re(\lambda)-E|$. From Figure~\ref{fig:unbalance_coulomb} bottom, it is clear that this true residual is substantially smaller than the one estimated 
by $|\Im(\lambda)|$. 
\end{remark}

\appendix
\section{Entries of the matrix polynomial coefficients}
\label{appendix:matr_entries}

The recursive identities satisfied by the Hermite functions allow us to find recursive expressions for the matrix entries of $K$ and $L$ in \eqref{eq_coeff_qpm} when $G=H_\kappa$. Rather than estimating the corresponding inner products by trapezoidal rules, we build the codes involved in 
the numerical experiments performed in Section~\ref{section:numerical} using these explicit expressions. As large factors are cancelled
in these explicit expressions, this approach turns out to be far more accurate. Since some of the calculations are not entirely trivial, we include here the crucial details. 

Let
\begin{gather*}
      T_{1}=\int_0^\infty \Phi_k(r)\Phi_j(r) \ud r, \qquad
      T_{2}(k,j)=\int_0^\infty \Phi_k'(r)\Phi_j(r) \ud r, \\
      T_{3}=\int_0^\infty \frac{1}{r}\Phi_k(r)\Phi_j(r) \ud r, \qquad 
      T_{4}=\int_0^\infty \Phi_k'(r)\Phi_j'(r) \ud r, \\
      T_{5}(k,j)=\int_0^\infty \frac{1}{r}\Phi_k'(r)\Phi_j(r) \ud r, \qquad 
      T_{6}=\int_0^\infty \frac{1}{r^2}\Phi_k(r)\Phi_j(r) \ud r, \\
    F_1=\int_0^\infty \phi_{\mathrm{el}}(r)\Phi_k(r)
   \Phi_j(r) \ud r, \qquad
    F_2=\int_0^\infty \phi_{\mathrm{el}}^2(r)\Phi_k(r)
   \Phi_j(r) \ud r, \\
    F_3=\int_0^\infty \frac{\phi_{\mathrm{el}}(r)}{r}\Phi_k(r)
   \Phi_j(r) \ud r, \qquad
    F_4(k,j)=\int_0^\infty \phi_{\mathrm{el}}(r)\Phi_k'(r)
   \Phi_j(r) \ud r. 
\end{gather*}
Here and below we stress the dependence on $j,k$ when the coefficient
is not symmetric with respect to these indices. 
Denote 
\[
    \Psi_{k,1}=\begin{pmatrix} \Phi_k \\ 0 \end{pmatrix}, \qquad
    \Psi_{j,2}=\begin{pmatrix} 0 \\ \Phi_j \end{pmatrix}.
\]
Then $\langle H_\kappa \Psi_{kl}, \Psi_{jm}\rangle$ are given according to Table~\ref{table:lin_coeff} and $\langle H_\kappa \Psi_{kl}, H_\kappa \Psi_{jm}\rangle$ are given according to Table~\ref{table:qua_coeff}.

\begin{table}[h!]
\centerline{\begin{tabular}{|c|c|c|}\hline  & $m=1$  & $m=2$ \\
\hline 
$l=1$ & $T_1+F_1$ & $T_2(k,j)+\kappa T_3$ \\
\hline
$l=2$ & $-T_2(k,j)+\kappa T_3$ & $-T_1+F_1$ \\
\hline
\end{tabular}}
\caption{Term $\langle H_\kappa \Psi_{kl}, \Psi_{jm}\rangle$.} \label{table:lin_coeff}
\end{table}
\begin{table}[h!]
\centerline{\begin{tabular}{|c|c|c|}\hline  & $m=1$  & $m=2$ \\
\hline 
$l=1$ & $T_1+T_4+\kappa T_5(k,j)+\kappa T_5(j,k)+$ & 
$-T_2(k,j)-T_2(j,k)+2\kappa F_3+$ \\ &
$\kappa^2T_6+2F_1+F_2$ & $F_4(k,j)-F_4(j,k)$ 
\\ \hline
$l=2$ & $-T_2(k,j)-T_2(j,k)+2\kappa F_3+$ & 
$T_1+T_4-\kappa T_5(k,j)-\kappa T_5(j,k)+$ \\ & $-F_4(k,j)+F_4(j,k)$ & 
$\kappa^2T_6-2F_1+F_2$ \\
\hline
\end{tabular}}
\caption{Term $\langle H_\kappa \Psi_{kl}, H_\kappa\Psi_{jm}\rangle$.} \label{table:qua_coeff}
\end{table}

For $m,n\in \N \cup {0}$, let
\[
    P(n)=\left\{ \begin{array}{ll}
    \prod_{l=1}^{n}\left(1+\frac{1}{2l}\right) &n\not=0 \\
    1 & n=0\end{array} \right. 
\]
and
\begin{equation} \label{eq:term_I}
    I(m,n)=\frac{1}{c_m c_n}\int_0^\infty h_m(r)h_n(r)e^{-r^2} \ud r.
\end{equation}

\begin{lemma}  \label{lemma:term_I}
\[
     I(m,n)=\left\{\begin{array}{ll} \delta_{mn} &
     m\equiv n\text{ (mod 2)} \\
     \frac{(-1)^{k-j+1}\sqrt{2P(k)P(j)}}{(2k-2j-1)\sqrt{\pi(2k+1)}}
     & m=2k,\,n=2j+1.   
     \end{array} \right.
\]
\end{lemma}
\begin{proof}
If $m\equiv n\text{ (mod 2)}$, then $h_m(r)h_n(r)$ is an even function
for $r\in \R$ and so
\[
   \int_0^\infty h_m(r)h_n(r)e^{-r^2} \ud r =\frac{1}{2}
   \int_{-\infty}^\infty h_m(r)h_n(r)e^{-r^2} \ud r.
\]  
On the other hand, if $m\not\equiv n\text{ (mod 2)}$, 
say $m=2k$ and $n=2j+1$, \eqref{eq:der} and integration by parts yield
\begin{align*}
   \int_0^\infty h_{2k}(r)h_{2j+1}(r)e^{-r^2} \ud r&=
   \int_0^\infty h_{2j+1}(r)(e^{-r^2})^{(2k)} \ud r\\
   &=-\int_0^\infty h_{2j+1}'(r)(e^{-r^2})^{(2k-1)} \ud r\\
   &=-2(2j+1) \int_0^\infty h_{2j}(r)(e^{-r^2})^{(2k-1)} \ud r\\
   &=2^2(j+1)(2k-1)\int_0^\infty h_{2k-2}(r)h_{2j-1}(r)e^{-r^2} \ud r.
\end{align*}
The corresponding expression for $I(m,n)$ can be obtained in a straightforward manner from these two assertions.
\end{proof}

\begin{lemma} \label{lemma:coeff_KL}
\begin{gather*}
    T_{1}(k,j)=\delta_{jk}, \qquad
    T_2(k,j)=\frac{4(-1)^{k-j+1}(k-j)}{\sqrt{\pi}(2k-2j-1)(2k-2j+1)}
     \sqrt{P(j)P(k)}, \\
    T_{3}(k,j)=\frac{2(-1)^{k-j+1}\sqrt{P(j)}}{\sqrt{\pi}\sqrt{P(k)}}\sum_{m=0}^k
    \frac{P(m)}{(2m+1)(2m-2j-1)},\\
    T_{4}(k,j)=\frac{1}{2}\left\{ \begin{array}{ll} -\sqrt{2k(2k+1)}& j=k-1 \\
           4k+3 &j=k \\  -\sqrt{(2k+2)(2k+3)} & j=k+1 \\
            0 &\text{otherwise},        \end{array} \right. \\
   T_{5}(k,j)=\left\{ \begin{array}{ll} 2(-1)^{j-k}
    \sqrt{\frac{P(k)}{P(j)}}& k<j \\
     1 & k=j \\ 0 & k>j, \end{array} \right. \qquad
   T_{6}(k,j)=(-1)^{j-k}2\left\{ \begin{array}{ll}
              \sqrt{\frac{P(j)}{P(k)}} & j\leq k \\
       \sqrt{\frac{P(k)}{P(j)}}& k< j.
   \end{array} \right.
\end{gather*}
\end{lemma}
\begin{proof}
Let $I(m,n)$ be given by \eqref{eq:term_I}.
By virtue of identities \eqref{eq:der} and \eqref{eq:recu}, 
\begin{align*}
\int_0^\infty h_{2k+1}'h_{2j+1}e^{-r^2} \ud r &=
(2k+1)I(2k,2j+1)-2^{-1}I(2k+2,2j+1), \\
\int_0^\infty \frac{1}{r}h_{2k+1}h_{2j+1}e^{-r^2} \ud r &=
\sum_{l=0}^{k}(-1)^l 2^{2l+1}\frac{k!}{(k-l)!}I(2(k-l),2j+1).
\end{align*}
This yields $T_2$ and $T_3$.

Let 
\[
   J(k,j)=\int_0^\infty\frac{1}{r} h_{2k}h_{2j+1}e^{-r^2} \ud r.
\]
Then
\[
    J(k,j)=\left\{ \begin{array}{ll} \frac{
    \sqrt{\pi} 2^{2j}(2k)!(-1)^{j-k}
    j!}{k!} & k\leq j \\
    0 & k>j \end{array}\right.
\]
and
\[
  \int_0^\infty \frac{1}{r}h_{2k+1}'h_{2j+1}e^{-r^2} \ud r =
   2(2k+1)J(k,j)-\delta_{kj}\sqrt{\pi}2^k(2k+1)!.  
\]
This renders $T_5$. Moreover, integration by parts ensures
\[
   T_6=\int_0^\infty \frac{1}{r^2}\Phi_k\Phi_j=-\int_{0}^\infty
   \Big(\frac{1}{r}\Big)'\Phi_k \Phi_j =
   \big(T_5(k,j)+T_5(j,k)\big).
\]

The expression for $T_4$ follows from \eqref{eq:eva_ho} and the identity
\[
   \int_0^\infty \Phi_k'(r)\Phi_j'(r) \ud r=
   \frac{1}{2}\int_{-\infty}^\infty -\Phi_k''(r)\Phi_j(r) \ud r.
\] 
\end{proof}

From $E_1$ and $E_2$ in the next lemma,
one easily obtains explicit formulae for
$F_n$ when $\phi_{\mathrm{el}}(r)=\gamma/r^\beta$.

\begin{lemma} \label{lemma:sub_cou}
For $\beta\in[0,1]$ and $\alpha\in[-1,2]$, let
\[
E_1(\beta,k,j)=\int_0^\infty \frac{1}{r^\beta}\Phi_k'(r)\Phi_j(r) \ud r, 
      \qquad E_2(\alpha,k,j)=\int_0^\infty \frac{1}{r^\alpha}\Phi_k(r)
   \Phi_j(r) \ud r.
\]
Then
\[
  E_1(\beta,k,j)=(2j+1)E_2(\beta+1,k,j)+\sqrt{2(2j+1)j}E_2(\beta+1,k,j-1)
   -E_2(\beta-1,k,j) 
\]
and 
\[
   E_2(\alpha,k,j)= \frac{2P(j)P(k)(-1)^{k+j}}{\sqrt{\pi}}
\sum_{m,p=1}^{k,j}\frac{(-1)^{m+p}\Gamma(\frac{3-\alpha}{2}+m+p)
\begin{pmatrix}k\\m \end{pmatrix}\begin{pmatrix}j\\p \end{pmatrix}}
{m!p!P(m)P(p)}.
\]
\end{lemma} 
\begin{proof} 
Let
\[
    S_n(k)=c_{2k+1}^{-1}(2k+1)!\frac{(-1)^{k-n}2^{2n+1}}{(k-n)!
    (2n+1)!}.
\]
Then
\[
    c_{2k+1}^{-1}h_{2k+1}(r)=\sum_{n=0}^k S_n(k)r^{2n+1}
\]
and
\[
    E_2(\alpha,k,j)= \sum_{m,p=1}^{k,j}S_m(k)S_p(j)K(\alpha,m,p)
\]
where
\[
    K(\alpha,m,p)=\int_0^\infty \frac{1}{r^{\alpha}}r^{2m+1}r^{2p+1}
    e^{-r^2}\ud r=\frac{1}{2}\Gamma \Big(\frac{3-\alpha}{2}+m+p\Big).
\]
On the other hand, the expression for $E_1$ follows from 
applying \eqref{eq:der} and \eqref{eq:recu}.
\end{proof}

If $E_3$ and $E_4$ are as in the following lemma and $\phi_{\mathrm{el}}(r)=\frac{1}{1+r^2}$,
then 
\begin{align*}
  F_1(k,j)&=E_3(2k+1,2j+1), \\
  F_2(k,j)&=\frac{1}{2}\Big(\sqrt{\frac{2k+1}{2}}E_4(2k,2j+1)+
  \sqrt{\frac{2j+1}{2}}E_4(2k+1,2j)  \\
  & \qquad -\sqrt{k+1}E_4(2k+2,2j+1)-\sqrt{j+1}E_4(2k+1,2j+2)+\\
  & \qquad E_3(2k+1,2j+1) \Big), \\
  F_3(k,j)&=T_3(k,j)-\sqrt{\frac{2k+1}{2}}E_3(2k,2j+1)-
  \sqrt{k+1}E_3(2k+2,2j+1), \\
  F_4(k,j)&=\sqrt{\frac{2k+1}{2}}E_3(2k,2j+1)-
  \sqrt{k+1}E_3(2k+2,2j+1).
\end{align*}

\begin{lemma}  \label{lemma:inv_harm}
For $m,n\in \N\cup\{0\}$, let $I(m,n)$ be as in Lemma~\ref{lemma:term_I},
\begin{align*} 
E_3(m,n)& =\frac{1}{c_m c_n}
\int_0^\infty \frac{1}{1+r^2}h_m(r)h_n(r)e^{-r^2} \ud r,\\ 
      E_4(m,n)& =\frac{1}{c_mc_n}
      \int_0^\infty \frac{r}{1+r^2}h_m(r)
   h_n(r) e^{-r^2} \ud r.
\end{align*}
Then 
\begin{align*}
    E_3(0,0)&=2e\int_0^1 e^{-t^2} \ud t+e\sqrt{\pi}, \\
    E_4(0,0)&=e\sqrt{\pi}\int_1^\infty \frac{e^{-t}}{t} \ud t, \\
    E_3(m+1,0)&=\frac{1}{\sqrt{m+1}}
    \big(E_4(m,0)-\sqrt{m}E_3(m-1,0)\big), \\
    E_4(m+1,0)&=\frac{1}{\sqrt{m+1}}
    \left(\sqrt{2}I(m,0)-\sqrt{2}E_3(m,0)-\sqrt{m}E_4(m-1,0) \right) ,\\
    E_3(m+1,n+1)&=\frac{1}{\sqrt{(m+1)(n+1)}}
    \big(2I(m,n)-2E_3(m,n)+ \\
    & \qquad -\sqrt{2m}E_4(m-1,n)-
    \sqrt{2n}E_4(m,n-1)+\sqrt{mn}E_3(m-1,n-1)\big) \\
    E_4(m+1,n+1)&=\frac{1}{\sqrt{(m+1)(n+1)}}
    \left(\sqrt{2(n+1)}I(m,n+1)-2E_4(m,n)+ \right. \\
    & \qquad \left. \quad \sqrt{2n}E_3(m,n-1)-
    \sqrt{2m}E_4(m-1,n)+\sqrt{mn}E_3(m-1,n-1)\right).
\end{align*}
\end{lemma}
\begin{proof}
The recursions for $E_3$ and $E_4$, follow from
\eqref{eq:recu}.
\end{proof}

\newpage
\begin{center}
{\bf Acknowledgements}
\end{center}
The authors would like to thank Jean Dolbeault, Mathieu Lewin, Michael Levitin and \'Eric S\'er\'e for fruitful discussions during the preparation of this work.

The first author is grateful for the hospitality of CEREMADE and Universit\'e de Franche-Comt\'e. The second author has been partially supported by ESPRC grant EP/D054621.

\bibliographystyle{alpha}
\bibliography{biblio}

\newpage

\begin{figure}[h!]  
\hspace{.5cm}{\includegraphics[width=15cm]{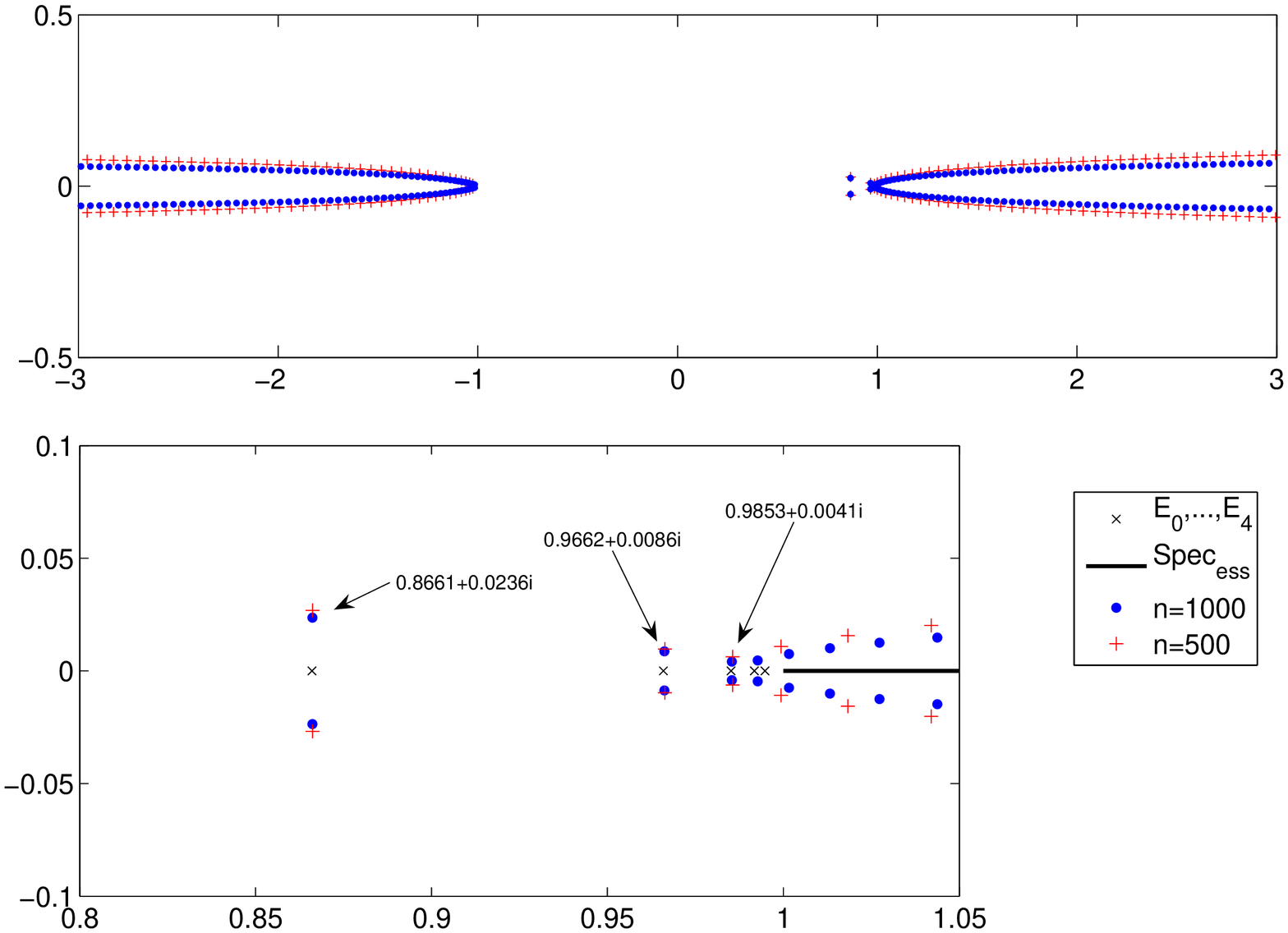}}
\caption{Portion of $\Spec_2(H_{-1},\mathcal{L}_n)$ in 
thin boxes around $[-3,3]$ for the purely coulombic potential 
with $\gamma=-1/2$. The bottom image shows details of the
picture near $1$. Here we superimpose two values of $n$, $500$ and $1000$. 
According to Theorem~\ref{non_pollution} there are approximate energy
states at
$E\approx 0.8661\pm 0.0236$, $E\approx 0.9662\pm 0.0086$ and
$E\approx 0.9853\pm 0.0041$. These correspond to the actual 
eigenvalues 
$E_0\approx 0.866025$, 
$E_1\approx 0.965925$ and 
$E_2\approx 0.985121$. }\label{fig:1}
\end{figure}

\newpage

\begin{figure}[h!]
\centerline{\includegraphics[height=8cm]{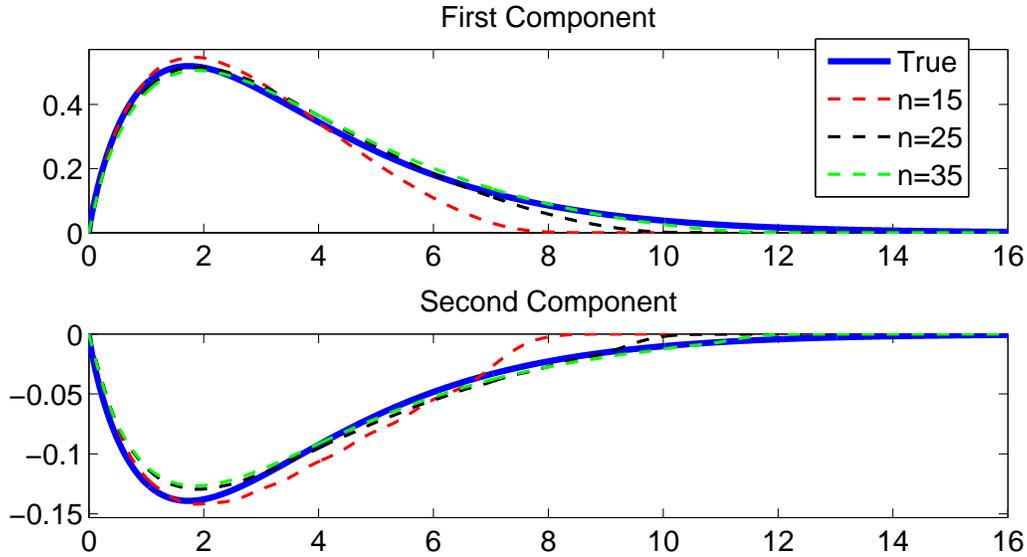}}
\caption{Approximate ground wave function for a
purely coulombic potential with $\gamma=-1/2$. The true wave function
(blue line) can be found explicitly, \cite[Section~7.4.2]{Thaller}.
The ground state in this case is $E_0=\sqrt{1-\gamma^2}\approx 0.866025$.
We have deliberately chosen small dimensions $n$ for the test spaces,
in order to illustrate approximation of the method.
The residual error is actually much smaller that the one
predicted by Theorem~\ref{non_pollution}, see 
Table~\ref{table:1}.}\label{fig:2}
\end{figure}

\begin{table}[h!]
\centerline{\begin{tabular}{crr} $n$ & $\frac{\|v-\Pi_{\mathcal{E}} v\|}{\|v\|}$  &
$\frac{|\Im(\lambda)|}{d_E}$ \\ \hline $15$ & $0.176115$ &
$0.680599$ \\ $25$ & $0.084527$ & $0.514205$ 
\\  $35$ & $0.072552$ & $0.457034$
\end{tabular}}
\caption{Here we compare both sides
of \eqref{eq_approx_efu}, for the computation of the approximate
eigenfunctions of Figure~\ref{fig:2}. 
We approximate $d_E=E_1-E_0\approx 0.0999004$.} \label{table:1}
\end{table}

\newpage

\begin{figure}[h!]
\centerline{\includegraphics[height=6cm]{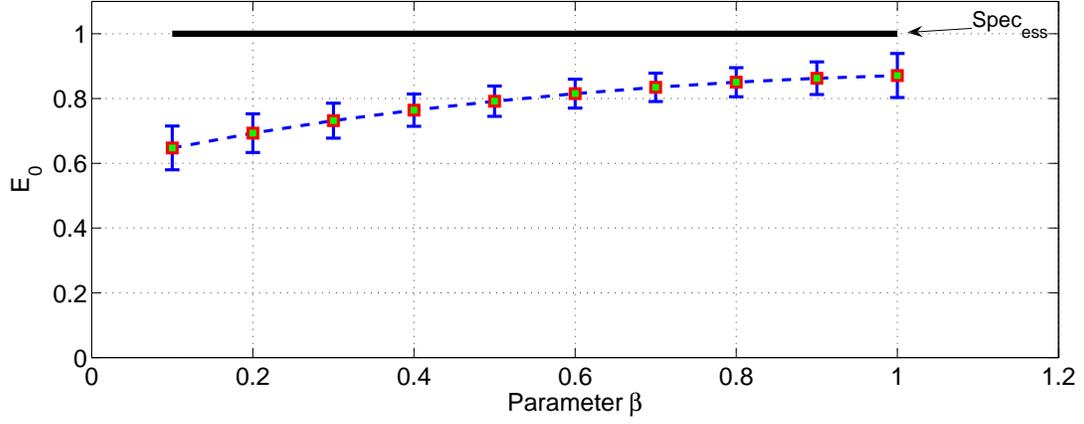}}
\caption{Computation of the ground energy value for $H_{-1}$ and
$\phi_{\mathrm{el}}(r)=-\frac{1}{2r^\beta}$. We depict $E_0$ against $\beta$.
The vertical bars correspond to the error predicted by 
Theorem~\ref{non_pollution}.}\label{fig:3}
\end{figure}
\begin{table}[h!]
\centerline{\begin{tabular}{crr} $\beta$ & $E_0$  &
$\Im(\lambda)$ \\ \hline 
    0.1  &  0.6474 &   0.0675 \\
    0.2  &  0.6932 &   0.0599 \\
    0.3  &  0.7316 &   0.0542 \\
    0.4  &  0.7642 &   0.0499\\
    0.5  &  0.7918 &   0.0468\\
    0.6  &  0.8151 &   0.0448\\
    0.7  &  0.8346 &   0.0439\\
    0.8  &  0.8505 &   0.0449\\
    0.9  &  0.8627 &   0.0504\\
    1.0  &  0.8711 &   0.0680
\end{tabular}}
\caption{Data depicted in Figure~\ref{fig:3}} \label{table:2}
\end{table}

\newpage

\begin{figure}[ht!]
\centerline{\includegraphics[height=10cm]{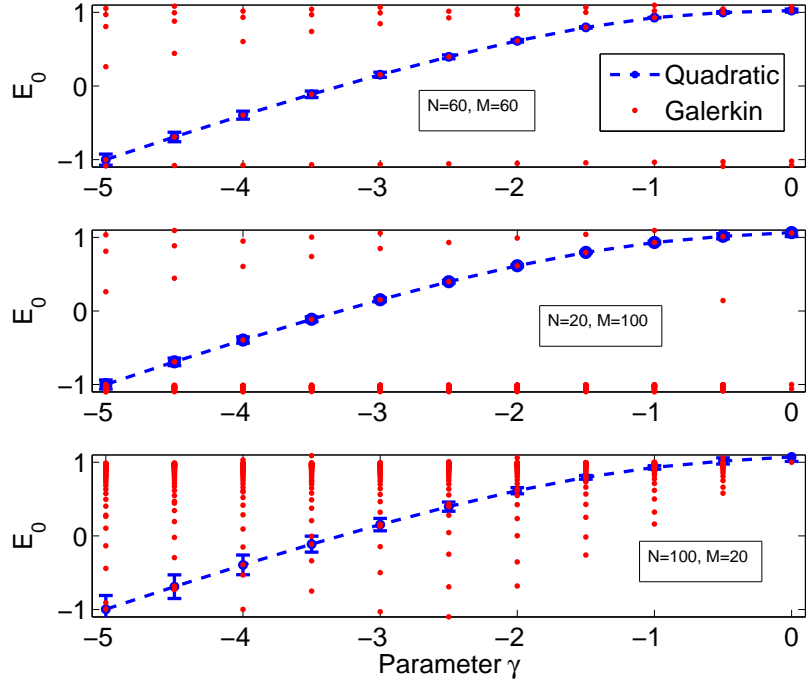}}
\caption{The graph captures the evolution of $E_0$ as it crosses the
spectral gap of $H_{-1}$ where
$\phi_{\mathrm{el}}(r)=\gamma/(1+r^2)$ for $\gamma=-5:.5:0$. We consider three choices of pairs $(N,M)$ such that $\dim(\mathcal{L}_{NM})=120$. The curve corresponds to $\Re(\lambda_n)$ for $\lambda=\lambda_n$ in \eqref{eq_approx_efu} and \eqref{convergence_rate}. The vertical bars on the curve measure $|\Im(\lambda_n)|$ . We superimpose the image with the eigenvalues of $L$ in
\eqref{eq_coeff_qpm} for $G=H_{-1}$, that is the Galerkin approximation.}
\label{fig:evo_smooth_gamma_neg}
\end{figure}

\newpage

\begin{figure}[ht!]
\centerline{\includegraphics[height=12cm]{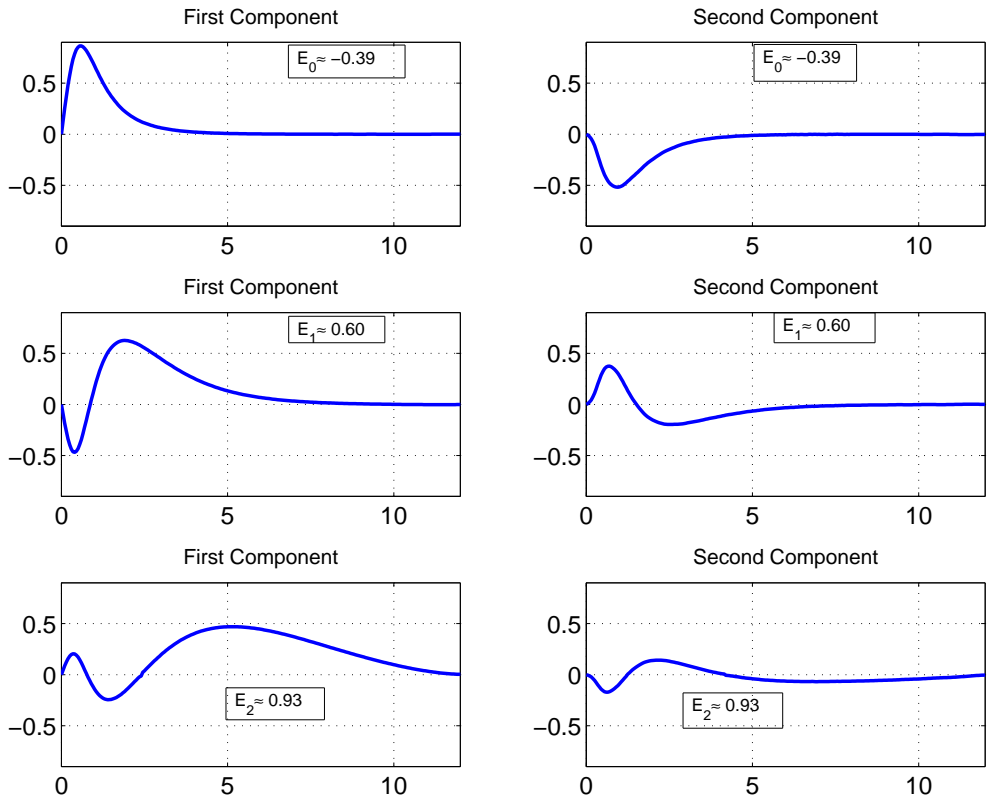}}
\caption{Here we use Theorem~\ref{non_pollution} to find the
first three eigenfunctions of $H_{-1}$ with
$\phi_{\mathrm{el}}(r)=-4/(1+r^2)$. The numerical evidence suggests:
$E_0\approx -0.3955$, $E_1\approx0.6049$ and $E_2\approx 0.9328$.}
\label{fig:eigenfu_smooth}
\end{figure}

\newpage

 \begin{figure}[ht!]
\centerline{\includegraphics[height=10cm]{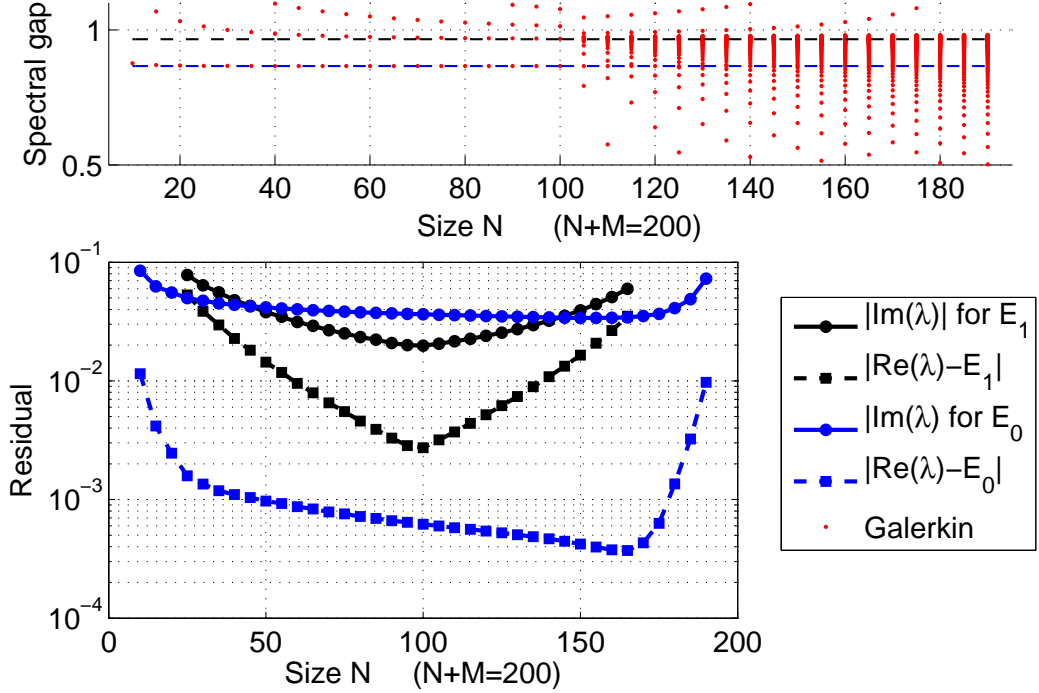}}
\caption{Here $E_0$ and $E_1$ are eigenvalues of $H_{-1}$ for 
$\phi_{\mathrm{el}}(r)=-\frac{1}{2r}$. 
The top graph shows the eigenvalues of $L$ in
\eqref{eq_coeff_qpm} (that is the Galerkin approximation) for $G=H_{-1}$ and  $(M,N)=(N,200-N)$ so that $\dim(\mathcal{L}_{NM})=200$.
The bottom graph depicts the residuals  $|\Im(\lambda_n)|$ and
$|\Re(\lambda)-E_j|$. For $E_0$, the minimum of the residual curve
corresponding to $|\Im(\lambda)|$ is achieved when $N\approx 155$ and it is roughly $7\%$ smaller than when $N=100$.  For the same eigenvalue,
the residual curve corresponding to $|\Re(\lambda)-E_0|$ achieves its minimum when $N=165$ and it is roughly $66\%$ smaller than when $N=100$.}
\label{fig:unbalance_coulomb}
\end{figure}

\newpage

 \begin{figure}[ht!]
\centerline{\includegraphics[height=10cm]{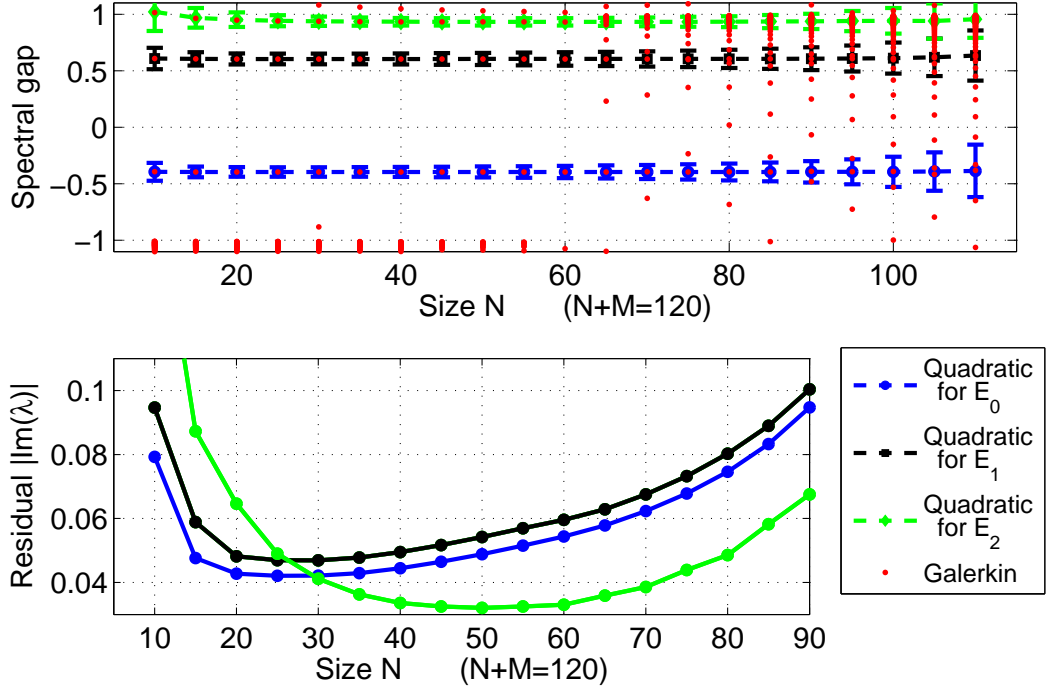}}
\caption{Here $E_0,\,E_1$ and $E_2$, are the first three eigenvalues of
$H_{-1}$ for $\phi_{\mathrm{el}}(r)=-4/(1+r^2)$. The top graph shows approximation of $E_0\approx -0.3955$, $E_1\approx0.6049$ and $E_2\approx 0.9328$, for  $(M,N)=(N,120-N)$ so that $\dim(\mathcal{L}_{NM})=120$. 
The curves correspond to $\Re(\lambda_n)$ for $\lambda=\lambda_n$ in \eqref{eq_approx_efu} and \eqref{convergence_rate}. The vertical bars measure $|\Im(\lambda_n)|$. The image is superimposed with the eigenvalues of $L$ in
\eqref{eq_coeff_qpm} for $G=H_{-1}$, that is the Galerkin approximation.
The bottom graph depicts the residuals  $|\Im(\lambda_n)|$.}
\label{fig:evo_unbalance}
\end{figure}

 \newpage

\begin{figure}[ht!]
\centerline{\includegraphics[height=7cm]{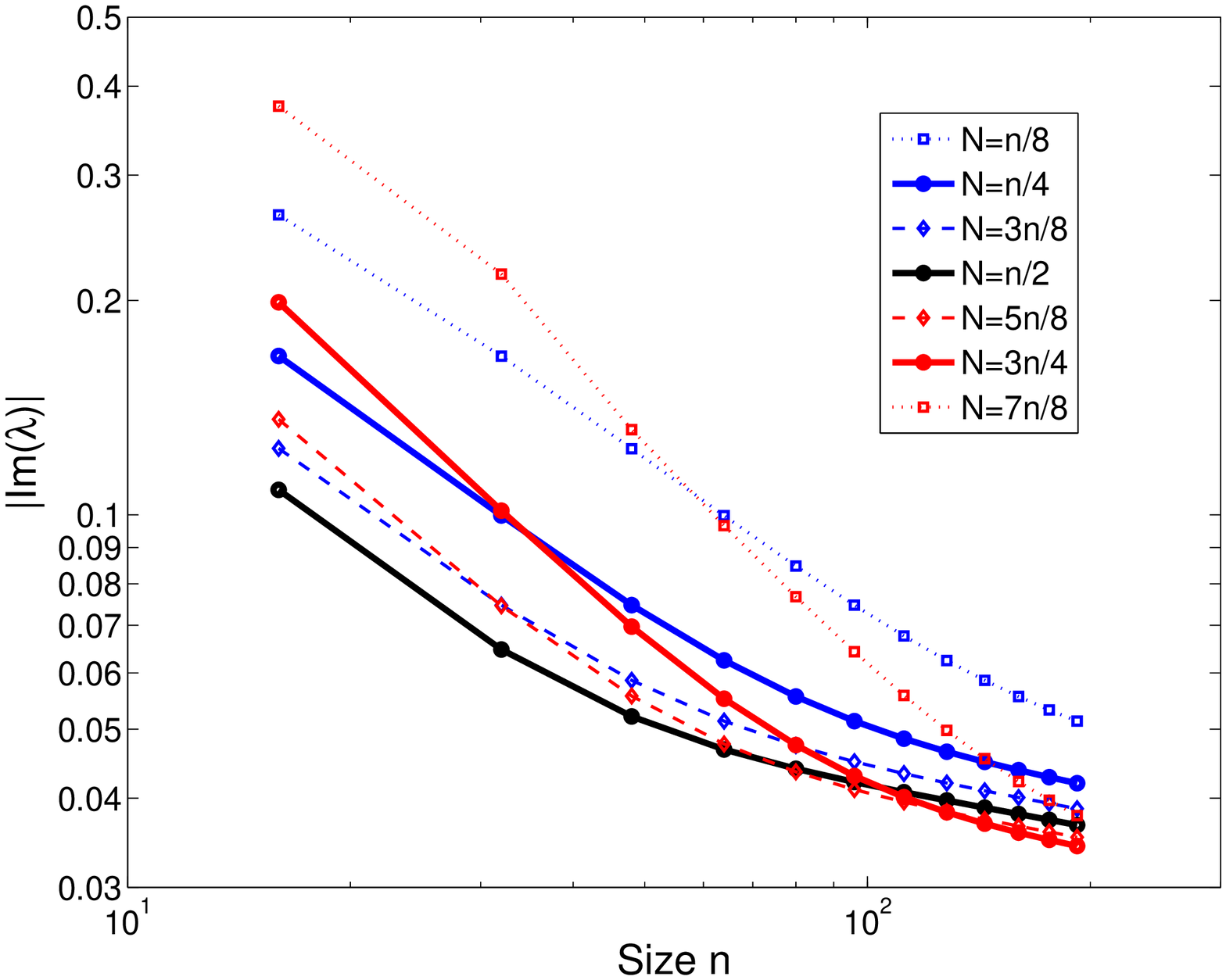}
\includegraphics[height=7cm]{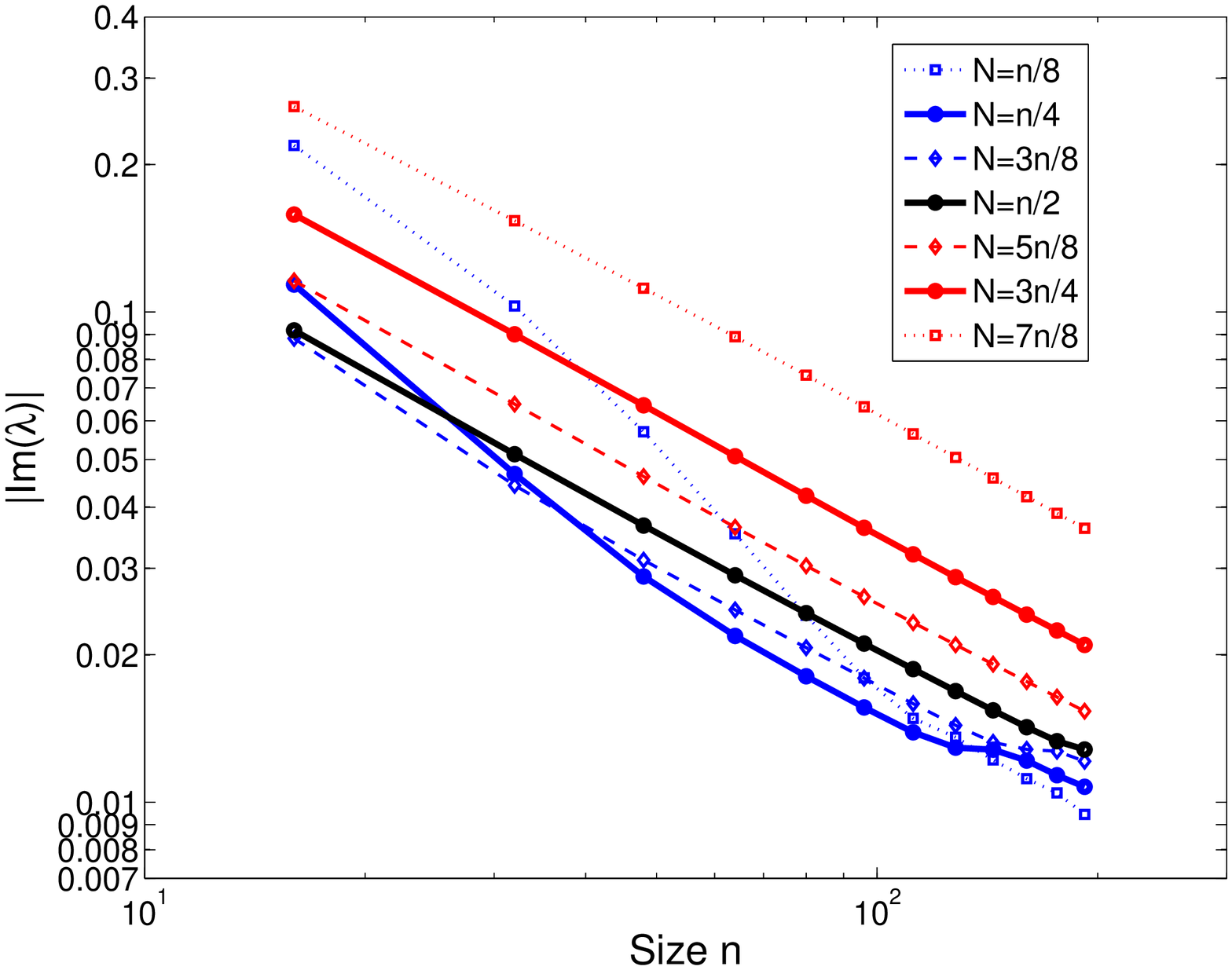}}
\caption{Log-log plots of $|\Im(\lambda)|$ for 
$\Re(\lambda)$ close to an eigenvalue, $E_0$, for different choices of pairs $(N,M)$ as $n=N+M$ increases. Left:
$\kappa=-1$, $\phi_{\mathrm{el}}(r)=-\frac{1}{2r}$ and $E_0\approx 0.86602$. Right: $\kappa=-1$,
$\phi_{\mathrm{el}}(r)=-\frac{2}{1+r^2}$ and $E_0\approx 0.61399$.
See Table~\ref{table:expo1}.}
\label{fig:rate_unbalance}
\end{figure}
\begin{table}[h!]
\centerline{\begin{tabular}{ccc} $N$ & $a$ &
$b$ \ \\ \hline 
    $n/8$ &  -0.6736  &  1.6766\\
    $n/4$ &  -0.5426  &  0.6555\\
    $3n/8$ &  -0.4385  &  0.3530\\
    $n/2$ &  -0.3963  &  0.2703\\
    $5n/8$ &  -0.5064  &  0.4478\\
    $3n/4$ &  -0.6903  &  1.1115\\
    $7n/8$ &  -0.9609  &  5.4520
\end{tabular}
\hspace{1cm}\begin{tabular}{ccc} $N$ & $a$ &
$b$ \ \\ \hline 
    $n/8$ &  -1.3241  &  8.8276  \\
    $n/4$ &  -0.9135  &  1.1303  \\
    $3n/8$ &  -0.7990  &  0.7223  \\
    $n/2$ &  -0.7979  &  0.8155  \\
    $5n/8$ &  -0.8125  &  1.0825  \\
    $3n/4$ &  -0.8163  &  1.5171  \\
    $7n/8$ &  -0.8004  &  2.4558
\end{tabular}}
\caption{In this table we fit by least squares the data of Figure~\ref{fig:rate_unbalance} and find $a$ and $b$ such that
$|\lambda_{n}-E_0|\leq |\Im(\lambda_{n})|\sim b n^{a}$ for $n=N+M$.} \label{table:expo1}
\end{table}

\end{document}